\documentclass[11pt,notitlepage,a4paper]{article}

\usepackage{amssymb}
\usepackage[intlimits]{amsmath}
\usepackage{amsfonts}
\usepackage{amsthm} 

\usepackage{mathptmx}

\usepackage{hyperref}
\hypersetup{colorlinks=true,linkcolor=RoyalPurple,citecolor=RoyalPurple}

\usepackage{cite}

\usepackage{graphicx}

\usepackage{geometry}

\usepackage{color}
\usepackage[dvipsnames]{xcolor}
\definecolor{antonio}{rgb}{.0,.5,.0}

\let\oldsqrt\sqrt
\def\sqrt{\mathpalette\DHLhksqrt}
\def\DHLhksqrt#1#2{%
\setbox0=\hbox{$#1\oldsqrt{#2\,}$}\dimen0=\ht0
\advance\dimen0-0.2\ht0
\setbox2=\hbox{\vrule height\ht0 depth -\dimen0}%
{\box0\lower0.4pt\box2}}

\allowdisplaybreaks

\newcommand{\R}{\mathbb{R}} 
\newcommand{\N}{\mathbb{N}} 
\newcommand{\dist}{\textnormal{dist}} 
\newcommand{\essinf}{\textnormal{essinf}} 

\renewcommand{\phi}{\varphi}

\newcommand{\cD}{{\mathcal D}}
\newcommand{\ccD}{{\mathcal {\bf D}}}
\newcommand{\cE}{{\mathcal E}}

\newcommand{\cH}{{\mathcal H}}

\newcommand{\cN}{{\mathcal N}}

\newcommand{\cV}{{\mathcal V}}
\newcommand{\ccV}{{\mathcal {\bf V}}}

\theoremstyle{definition}

\theoremstyle{plain} 
\newtheorem{defi}{Definition}[section]
\newtheorem{thm}[defi]{Theorem}
\newtheorem{prop}[defi]{Proposition}
\newtheorem{lemma}[defi]{Lemma}
\newtheorem{cor}[defi]{Corollary}
\newtheorem{remark}[defi]{Remark}

\theoremstyle{definition}

\title{Foliated Schwarz symmetry of solutions to a cooperative system of equations involving nonlocal operators}
\author{
Antonio Greco\footnote{Dipartimento di Matematica e Informatica, via Ospedale 72, I-09124 Cagliari, greco@unica.it} \ and\ 
Sven Jarohs\footnote{Institut f\"ur Mathematik, Goethe-Universit\"at, Frankfurt, Robert-Mayer-Stra\ss e 10, D-60629 Frankfurt,\newline jarohs@math.uni-frankfurt.de.}
}
\date{\today}

\begin{document}
\maketitle

\pdfbookmark[1]{Abstract}{Abstract}
\begin{abstract}
In this paper, we prove foliated Schwarz symmetry of solutions to a cooperatively coupled system of equations involving nonlocal operators. Here, the class of nonlocal operators covers in particular the case of the fractional Laplacian. Moreover, we give an explicit example of a nonlocal nonlinear system, in which our result can be applied.
\end{abstract}
{\footnotesize
\begin{center}
\textit{Keywords.} Nonlocal Operator $\cdot$ Axial symmetry $\cdot$ Maximum Principle for Systems
\end{center}
\begin{center}
%
\end{center}
}

\section{Introduction}
In the following, we investigate symmetry properties of solutions of a system of equations in radial domains. More precisely, for $m\in \N$ we investigate bounded continuous solutions of
\begin{equation}\label{eq:system-basis}
\left\{\ \begin{aligned}
Iu_i&= f_i(|x|,u_1,\ldots, u_m) &&\text{in $\Omega$, $i=1,\ldots,m$}\\
u_1&=\ldots=u_m=0 &&\text{in $\R^N\setminus \Omega$}
\end{aligned}\right.
\end{equation}
where $\Omega\subset\R^N$ is a bounded radial domain. Moreover, $f_1,\ldots,f_m\in C^1([0,\infty)\times\R^m)$ are nonlinearities to be specified later and $I$ is a nonlocal operator, which for $u\in C^2(\R^N)\cap L^{\infty}(\R^N)$ is given by
\begin{equation}\label{defi-op}
Iu(x):=p.v.\int_{\R^N}(u(x)-u(y)) \, k(x-y)\ dy:=\lim_{\epsilon\to0^+}\int_{\R^N\setminus B_{\epsilon}(0)}(u(x)-u(y)) \, k(x-y)\ dy,\quad x\in \R^N.
\end{equation}
Here $k:\R^N\to[0,\infty]$ is a kernel function, which is given by $k(z)=k_0(|z|)$, $z\in \R^N$ for a monotone decreasing function $k_0:[0,\infty)\to[0,\infty]$ satisfying
\begin{equation}\label{assumption-kernel}
\int_0^{\infty}\min\{1,r^2\} \, k_0(r) \, r^{N-1}\ dr<\infty\quad\text{and}\quad \int_{0}^{\infty}k_0(r) \, r^{N-1}\ dr=\infty.
\end{equation}
We note that these assumption on $k$ in particular cover the case, where $I=(-\Delta)^s$, $s\in(0,1)$ by setting $k_0(r)=c_{N,s} \, r^{-N-2s}$ for a normalization constant $c_{N,s}>0$ given by
\begin{equation}\label{C}
c_{N,s}
=
\frac{\, 2^{2s} s \, \Gamma(\frac{\, N \,}{2} + s) \,}
{\pi^\frac{\, N \,}{2} \, \Gamma(1-s)}
=
\Big(\int_{\mathbb R^N}
\frac{\, 1-\cos x_1 \,}{|x|^{N+2s}}\,dx\Big)^{\!-1}
.
\end{equation}
The value of $c_{N,s}$ is chosen to make the fractional Laplacian the pseudodifferential operator whose symbol is $|\xi|^{2s}$ (see e.g. \cite[Section 3.1]{CS} or\cite[Proposition~3.3]{DPV} for details; the equality of the two values given in~\eqref{C} is shown in \cite{F13,J15}, see also \cite{GM}). For further information on the operator~$I$ and the definition of \textit{weak solution}, which we use in this paper, we refer to Subsection~\ref{bilinearform} below, see also \cite{FKV15,JW16}.

\bigskip
Symmetry properties of solutions to nonlocal nonlinear problems have been studied for one or more equations in the case where $I$ is the fractional Laplacian in \cite{FQT12,FW13-2,JW14,GM}, while in \cite{JW16} the question of symmetries to solutions was studied for a general class of nonlocal operators. However, if $\Omega$ is not a ball but rather an annulus or of the solutions change sign, then in general it is no longer true that a solution of \eqref{eq:system-basis} must be radial even in the case where $I$ is a local operator and $m=1$. However, under some suitable assumptions on the equation or the system, some axial symmetry can still be achieved. In the case where $m=1$ and $I$ is a local operator this has been studied in \cite{P02,PW07,SW12}, whereas symmetry for systems have been studied in \cite{DP13,DP19} (see also there references in there). For the nonlocal case, the axial symmetry of solutions has been studied in \cite{J16} for $m=1$. 
\bigskip

In the following, we are interested in solutions of \eqref{eq:system-basis}, which are not radial and possibly change sign. For this we consider a particular kind of axial symmetry called \textit{foliated Schwarz symmetry}, which was defined in~\cite[Definition~2.4]{SW03}, based on an idea in~\cite{Polya}. We also refer to the general survey -- in particular Section 2.3 -- in \cite{W09}.

\bigskip

Let $\Omega\subset\R^N$, $N\geq 2$ be a radial domain, $p\in S^{N-1}:=\{x\in \R^N\;:\; |x|=1\}$. A function $u\colon\Omega\to \R$ is called \textit{foliated Schwarz symmetric with respect to $p$ in $\Omega$}, if for every $r>0$ with $re_1\in D$ and $c\in \R$, the restricted superlevel set $\{x\in r \, S^{N-1}\;:\; u(x)\geq c\}$ is equal to $r \, S^{N-1}$ or a geodesic ball in the sphere $r \, S^{N-1}$ centered at $rp$.

We simply call $u$ \textit{foliated Schwarz symmetric}, if $u$ has this property for some unit vector $p\in \R^N$.

\bigskip

We give an equivalent definition in Section \ref{foliated} (see also \cite[Proposition 3.3]{SW12}) below, which we use in our proof. Note that if $u\colon\R^N\to \R$ is such that $u|_\Omega$ is foliated Schwarz symmetric with respect to some $p$ for some radial set $\Omega\subset \R^N$, then $u\,\chi_\Omega$ is axially symmetric with respect to the axis $\R\cdot p$ and nonincreasing in the polar angle $\theta=\arccos (\frac{x}{|x|}\cdot p)$.

Our main result on the symmetry properties of solutions of \eqref{eq:system-basis} is the following.

\begin{thm}\label{thm:main1}
Let\/ $\Omega\subset\R^N$ be a bounded radial domain and $m\in \N$. Assume that $f_i\in C^1([0,\infty) \allowbreak \times \R^m,\R^m)$, $(r,u_1,\ldots,u_m)\mapsto f_i(r,u_1,\ldots,u_m)$, $i=1,\ldots,m$ satisfies
\begin{equation}\label{thm:main1-assumption1}
\frac{\partial}{\partial u_i} \, f_j> 0 \quad\text{on $[0,\infty)\times\R^m$, $i,j\in\{1,\ldots,m\}$, $i\neq j$.}
\end{equation}
Let\/ $u_1,\ldots,u_m:\R^N\to \R$ be continuous bounded functions satisfying\/ \eqref{eq:system-basis} in weak sense. If 
\begin{equation}\label{thm:main1-assumption2}
\begin{split}
u_i(x_1,x')\geq u_i(-x_1,x')\quad\text{ for all $x_1>0$, $x'\in \R^{N-1}$, $i=1,\ldots,m$, and}\\
\text{there is $x_1>0$, $x'\in \R^{N-1}$ with}\quad u_1(x_1,x')>u_1(-x_1,x'),
\end{split}
\end{equation}
then there is $p\in S^{N-1}$ such that $u_1,\ldots,u_m$ are foliated Schwarz symmetric in $\Omega$ with respect to $p$ and strictly decreasing in the polar angle.
\end{thm}

Clearly, assumption \eqref{thm:main1-assumption1} is restricting, which is due to the fact that we did not assume any further connection between the $f_i$ and $u_j$.  In the following variant of Theorem \ref{thm:main1} we weaken the assumption on $f_i$ but assume positivity of the $u_i$.

\begin{thm}\label{thm:main2}
Let\/  $\Omega\subset\R^N$ be a bounded radial domain and $m\in \N$. Assume that $f_i\in C^1([0,\infty) \allowbreak \times \R^m,\R^m)$, $(r,u_1,\ldots,u_m)\mapsto f_i(r,u_1,\ldots,u_m)$, $i=1,\ldots,m$ satisfies
\begin{equation}\label{thm:main1-assumption1c}
\frac{\partial}{\partial u_i}f_j> 0 \quad\text{on $(0,\infty)^{m+1}$, $i,j\in\{1,\ldots,m\}$, $i\neq j$.}
\end{equation}
Let\/ $u_1,\ldots,u_m\colon\R^N\to \R$ be continuous bounded functions satisfying \eqref{eq:system-basis} in weak sense. If $u_i>0$ in $\Omega$ for $i=1,\ldots,m$ and
\eqref{thm:main1-assumption2} holds, then there is $p\in S^{N-1}$ such that $u_1,\ldots,u_m$ are foliated Schwarz symmetric in $\Omega$ with respect to $p$ and strictly decreasing in the polar angle.
\end{thm}

The proof of Theorem \ref{thm:main1} and \ref{thm:main2} follows from the rotating plane method for nonlocal operators developed in \cite{J16}, which we adjust here to the case of systems. Theorem \ref{thm:main1} and \ref{thm:main2} then follow from Theorem \ref{thm:main1-final} below.

\begin{remark} Some remarks are necessary on the assumptions in Theorem \ref{thm:main1} and \ref{thm:main2}.
\begin{enumerate}
\item Assumption \eqref{thm:main1-assumption1} or assumption \eqref{thm:main1-assumption1c} with the positivity assumption on the $u_i$ can be weakened further. Indeed, we only need a kind of \textit{strongly coupled} condition on a linearized system connected with \eqref{eq:system-basis} (see Section \ref{mp} below).
\item Assumption \eqref{thm:main1-assumption2} clearly can be rotated to an arbitrary hyperplane with respect to which the first inequality holds. Moreover, it is enough to assume the second assumption on any $u_i$ instead of $u_1$. For the general formulation, see Theorem \ref{thm:main1-final} below.
\item We note that the connectedness of $\Omega$ is not necessary if $k_0$ is strictly decreasing. Indeed, this assumption follows from the kind of strong maximum principle used in our proof (see also Remark \ref{connectedness}).
\end{enumerate}
\end{remark}

It is not obvious under which circumstances a solution $u_1,\ldots,u_m$ can be found such that \eqref{thm:main1-assumption2} is satisfied. In the following, we give an explicit example, covering the case of the fractional Laplacian, where Theorem \ref{thm:main2} can be applied. For this, we recall the first eigenvalue of the operator $I$ in an open subset $\Omega$ of $\R^N$ given by
\begin{equation}\label{first-eigenvalue}
\lambda_1(\Omega):=\inf_{u\in \cD_k(\Omega)\setminus \{0\}}\frac{\cE_k(u,u)}{\|u\|_{L^2(\Omega)}^2}.
\end{equation}
Recall from \cite{JW19} that $\lambda_1(\Omega)>0$ if $\Omega$ is bounded in one direction. The following existence statement is related to the Br\'ezis-Nirenberg problem, which for systems with the fractional Laplacian has been studied in \cite{FMPSZ16}. In the following statement, we consider a more general class of nonlocal operators, which includes the fractional Laplacian and deals with the geometry of the pair of solutions to the system.

\begin{thm}\label{thm:main3}
	Let the function $k_0$ in \eqref{assumption-kernel} be strictly decreasing and, for some $0<s\leq \sigma<1$, $\gamma\in(0,1)$, and $c>0$, satisfy
	\begin{equation}\label{thm3:assumption1}
	\frac1{\, c \,} \, r^{-1-2s} \leq k_0(r)\leq cr^{-1-2\sigma} \quad\text{for $r\in(0,1)\quad$and}\quad k_0(r)\leq cr^{-1-2\gamma}\quad\text{for $r\geq 1$.}
	\end{equation}
	Furthermore, let\/ $\Omega\subset\R^N$ be a bounded radial domain, $a_1,a_2\in L^{\infty}(\R)$ with $\|a^+_i\|_{L^{\infty}(\R)}<\lambda_1(\Omega)$ for $i=1,2$, and let $1<q<\frac{N}{N-2s}$. If $a_1\neq a_2$, then there are two continuous bounded functions $u_1,u_2\colon\R^N\to \R$, $u_1\neq u_2$, which are positive in\/ $\Omega$ and satisfy in weak sense
\begin{equation}\label{eq:system-basis2}
\left\{\ \begin{aligned}
Iu_1&=a_1(|x|) \, u_1+ |u_2|^q \, |u_1|^{q-2} \, u_1 &&\text{in $\Omega$}\\
Iu_2&=a_2(|x|) \, u_2+|u_1|^q \, |u_2|^{q-2} \, u_2 &&\text{in $\Omega$}\\
u_1&=u_2=0 &&\text{in $\R^N\setminus \Omega$.}
\end{aligned}\right.
\end{equation}
	Moreover, $u_1$ and $u_2$ are foliated Schwarz symmetric with respect to some $p\in S^{N-1}$ and, if $u_1$ and $u_2$ are not radial, then they are strictly decreasing in the polar angle.
\end{thm}

The paper is organized as follows. In Section \ref{sec:preliminaries} we present our notation and recall known statements on the nonlocal operators we use. Moreover, we introduce the notation for systems and recall the properties and definitions of Foliated Schwarz symmetry. In Section \ref{mp} we state and prove variants of maximum principles, which we use in Section \ref{sec:proof-main-thm} to prove Theorem \ref{thm:main1}. The proof of Theorem \ref{thm:main2} can be found in Section \ref{application}.

\section{Notation and Preliminaries}\label{sec:preliminaries}

In the following we use $N\in \N$ to denote the dimension. For $A,B\subset \R^N$ nonempty measurable sets we denote by $\chi_A: \R^N \to \R$ the characteristic function and $|A|$ the Lebesgue measure. The notation $B \subset \subset A$ means that $\overline B$ is compact and contained in the interior of $A$. We denote $\dist(A,B):=\inf_{a\in A,\ b\in B}|a-b|$ and as usual $\dist(\{x\},A):=\dist(x,A)$ for $x\in \R^N$. For $r>0$ we denote $B_r(A):=\{x\in \R^N\;:\; \dist(x,A)<\infty\}$ and then $B_r(x)$ denotes the ball of radius $r$ for $x \in \R^N$. Moreover, we fix $S^{N-1}:=\partial B_1(0)=\{x\in \R^N\;:\;|x|=1\}$ to denote the $N$-dimensional sphere.

As usual, for $A$ open, $C^m(A)$ (resp. $C^m(\overline{A})$) denotes the space of $m$-times continuously differentiable functions in $A$ (resp. $\overline{A}$) and $C^{0,1}(A)$ denotes the space of Lipschitz functions. $C^m_c(A)$ and $C^{0,1}_c(A)$ denotes respectively those functions in $C^m(A)$ or $C^{0,1}(A)$, which have compact support in $A$. In the following, if $X(A)$ is some function space and $u\in X(A)$ is a function, we always mean that $u\colon\R^N\to \R$ is such that $\chi_A \, u\in X(A)$ and $\chi_{\R^N\setminus A} \, u\equiv 0$. For instance, if $u\in L^2(A)$, then $u\in L^2(\R^N)$ and $u=0$ on $\R^N\setminus A$. 

Finally, for a function $u\colon A \to \R$ we use $u^+:=u_+:= \max\{u,0\}$ and $u^-:=-\min\{u,0\}$ to denote the positive and negative part of $u$ respectively, so that $u=u^+-u^-$.

\subsection{On the operator and associated spaces}\label{bilinearform}

Let $k\colon\R^N\to[0,\infty)$ be a radial and radial decreasing function. That is $k(z)=k_0(|z|)$, $z\in \R^N$ for a monotone decreasing function $k_0\colon[0,\infty)\to[0,\infty]$ satisfying \eqref{assumption-kernel}. We denote formally the bilinear form associated to $k$ by
\[
\cE_k(u,v):=\frac1{\, 2 \,}\int_{\R^N}\int_{\R^N} (u(x)-u(y)) \, (v(x)-v(y)) \, k(x-y)\ dy.
\]
For $\Omega \subset \R^N$ open, this bilinear form is well-defined on
\[
\cD_k(\Omega):=\{u\in L^2(\Omega)\;:\; \cE_k(u,u)<\infty\}.
\]
It follows that $\cD_k(\Omega)$ is a Hilbert space with scalar product
\[
\langle u,v\rangle_k:= \langle u,v\rangle_2 +\cE_k(u,v),\quad u,v\in \cD_k
\]
where $\langle\cdot,\cdot\rangle_2$ denotes the usual $L^2$ scalar product. By standard methods (see e.g. \cite{JW16,JW18}) it follows that $\cE_k$ is associated to a (nonlocal) operator $I$, which on $C^2(\R^N)\cap L^{\infty}(\R^N)$ is represented by \eqref{defi-op} and it holds
\[
\langle Iu,v\rangle_2=\int_{\R^N}Iu(x) \, v(x)\ dx=\cE_k(u,v)\quad\text{ for all $u\in C^2(\R^N)\cap L^{\infty}(\R^N)$, $v\in \cD_k(\R^N)$.}
\]
We note that the embedding $\cD_k(\Omega)\to L^2(\Omega)$ is locally compact in the sense that $\cD_k(\R^N)\ni u\mapsto \chi_B \, u \in L^2(\R^N)$ is compact for any bounded open set $K\subset \R^N$ (see \cite[Theorem 1.1]{JW19}). In the particular case, where $\Omega$ is bounded in one direction, say $\Omega\subset(-a,a)\times \Omega$ for some $a>0$, we have (see \cite[Proposition 1.7]{JW19}, \cite[Lemma 2.7]{FKV15})
\begin{equation}\label{eigenvalue0}
\lambda_1(\Omega):=\inf_{u\in \cD_k(\Omega)\setminus\{0\}}\ \frac{\cE_k(u,u)}{\|u\|_{L^2(\Omega)}^2}>0
\end{equation}
and moreover (see \cite[Proposition 1.7]{JW19}, \cite[Lemma 2.1]{JW16})
\begin{equation}\label{eigenvalue}
\lambda_1(\Omega)\to \infty\quad\text{ for either $a\to 0$ or $|\Omega|\to 0$.}
\end{equation}
It hence follows that in this case $\cE_k$ is a scalar product and the induced norm is equivalent to $\langle\cdot,\cdot\rangle_2$. In particular, if $\Omega$ is bounded, then $\cD_k(\Omega)\to L^2(\Omega)$ is compact and $\lambda_1(\Omega)$ corresponds to the first eigenvalue of $I$.

\bigskip

In the following, we understand solutions in the weak sense, that is, given $f\in L^{2}(\R^N)$, we say that $u\in \cD_k(\Omega)$ is a solution of
\begin{equation}\label{def-sol}
Iu=f\quad\text{in $\Omega$}\quad\text{and}\quad u=0 \quad\text{ in $\R^N\setminus \Omega$}
\end{equation}
if for all $\phi\in \cD_k(\Omega)$ we have
\[
\cE_k(u,\phi)=\int_{\Omega}f(x) \, \phi(x)\ dx.
\]
In particular, $u_1,u_2$ are called \textit{weak solution} of \eqref{eq:system-basis}, if for $i=1,2$, we have $u_i\in \cD_k(\Omega)$ and
\[
\cE_k(u_i,\phi)=\int_{\Omega}f_i(|x|,u_1,u_2) \, \phi(x)\ dx
\]
for all $\phi\in \cD_k(\Omega)$, whenever the right-hand side is well defined.

\bigskip\goodbreak

Finally, in our analysis, we use the \textit{rotating plane method} and linearize the system of equations. Our symmetry results then follow from an application of different maximum principles for supersolutions. For this, we extend the definition of $\cD_k(\Omega)$. Let $\Omega\subset\R^N$ open and denote 
\[
\cV_k(\Omega):=\Bigg\{ u\colon\R^N\to \R\text{ measurable}\;:\; \rho_k(u,\Omega):=\int_{\Omega}\int_{\R^N}(u(x)-u(y))^2 \, k(x-y)\ dx \, dy<\infty\Bigg\}.
\]
Clearly by definition we have for $A\subset B\subset \R^N$ open
\[
\cD_k(A)\subset\cD_k(B)\subset\cD_k(\R^N)\subset\cV_k(\R^N)\subset\cV_k(B)\subset\cV_k(A).
\]
The following Lemma collects all information on $\cV_k(\Omega)$ needed in this paper.
\begin{lemma}[\cite{JW16}, Lemma 3.1, Lemma 3.2 \cite{JW18}]\label{vk-properties}
Let $\Omega\subset \R^N$ open.
\begin{enumerate}
\item[1.] $\cE_k$ is well-defined on $\cV_k(\Omega)\times \cD_k(\Omega)$ and
\[
\cE_k(u,v)\leq (2+\sqrt{2}) \, \rho_k(u,\Omega)^{\frac{1}{2}} \, \cE_k(v,v)^{\frac{1}{2}}\quad\text{for $u\in \cV_k(\Omega)$, $v\in \cD_k(\Omega)$.}
\]
\item[2.] $u\in \cV_k(\Omega)$ implies $u^{\pm},|u|\in \cV_k(\Omega)$.
\item[3.] If $u\in \cD_k(\Omega)$, then $\cE_k(u^+,u^-)$ is well defined and
\[
\cE_k(u^+,u^-)\leq 0\quad\text{ and also }\quad \cE_k(|u|,|u|)\leq \cE_k(u,u).
\]
Moreover, if $k_0$ is strictly decreasing, then equality holds in these inequalities if and only if $u=u^+$ or $u=u^-$ a.~e. in $\R^N$
\end{enumerate}
Furthermore, if\/ $\Omega$ is in addition bounded and $u\in \cV_k(\Omega)$, then
\begin{enumerate}
\item[4.] $u \, \chi_{\R^N\setminus \Omega}\equiv 0$, then $u\in \cD_k(\Omega)$.
\item[5.] $u\geq 0$ on $\R^N\setminus \Omega$, then $u^-\in \cD_k(\Omega)$.
\end{enumerate}
\end{lemma}
The additional assertion in Lemma \ref{vk-properties}.3 follows immediately from the proof in \cite{JW16}.
Based on Lemma \ref{vk-properties} we say $u\in \cV_k(\Omega)$ satisfies for some $f\in L^2(\Omega)$ in weak sense
\begin{equation}\label{def-supersol}
Iu\geq f\quad\text{in $\Omega$}\quad\text{and}\quad u\geq 0 \quad\text{ in $\R^N\setminus \Omega$}
\end{equation}
if $u\geq 0$ on $\R^N\setminus \Omega$ and for all $v\in \cD_k(\Omega)$, $v\geq 0$ we have
\[
\cE_k(u,v)\geq \int_{\Omega} f(x) \, v(x)\ dx.
\]
We also call $u$ in this case a \textit{supersolution} of \eqref{def-sol}. Similarly, we call $u$ a \textit{subsolution} of \eqref{def-sol} if $-u$ satisfies in weak sense \eqref{def-supersol}.

\subsubsection{On the notation for systems}

In the following, let $M$ be any set, $m\in\N$, and $\Psi\colon M\to \R^m$, where we denote the coordinates of $\Psi$ with $\psi_1,\ldots,\psi_m:M\to \R$. We say $\Psi\geq0$ (or $>0$), if $\psi_i\geq0$ (or $>0$) for $i=1,\ldots,m$ and we say $\Psi\gneq0$, if $\Psi\geq0$ and there is $x\in M$ and $i\in\{1,\ldots,m\}$ such that $\psi_i(x)>0$. Furthermore, we denote $\Psi^\pm:=(\psi_1^{\pm},\ldots, \psi_m^{\pm})$.

\bigskip

We denote for $\Omega\subset\R^N$ open
\[
\ccD_k(\Omega):=\Big(\cD_k(\Omega)\Big)^{\! m}\quad \text{and}\quad \ccV_k(\Omega):=\Big(\cV_k(\Omega)\Big)^{\! m}.
\]

For $U=(u_1,\ldots,u_m)\in \ccV_k(\Omega)$, $V=(v_1,\ldots,v_m)\in \ccD_k(\Omega)$ we write
\[
\cE_k(U,V):=\sum_{i=1}^m\cE_k(u_i,v_i).
\]
and similarly, for $U\in (L^2(\Omega))^m$, $\|U\|_{L^2(\Omega)}^2=\sum\limits_{i=1}^{m}\int_{\Omega}(u_i)^2\ dx$.

Hence a solution $u_1,\ldots, u_m\in \cD_k(\Omega)$ of
\begin{equation}\label{eq:system1}
\left\{\ \begin{aligned}
Iu_i&= f_i(|x|,u_1,\ldots,u_m) &&\text{in $\Omega$}\\
u_i&=0 &&\text{in $\R^N\setminus \Omega$}
\end{aligned}\right.
\end{equation}
for $i=1,\ldots,m$, where $f_1,\ldots,f_m\in C^1([0,\infty)\times \R^m)$ can be rewritten in one equation by setting $U=(u_1,\ldots,u_m)\in \ccD_k(\Omega)$ and $F(r,U)=\big(f_i(r,u_1,\ldots,u_m)\big)_{1\leq i\leq m}$. The system \eqref{eq:system1} then reads

\begin{equation}\label{eq:system2}
\left\{\ \begin{aligned}
IU&= F(|x|,U) &&\text{in $\Omega$}\\
U&=0 &&\text{in $\R^N\setminus \Omega$.}
\end{aligned}\right.
\end{equation} 
and $U$ solves \eqref{eq:system2} in the weak sense if for all $V\in \ccD_k(\Omega)$ we have
\[
\cE_k(U,V)=\int_{\Omega} F(|x|,U(x))\cdot V(x)\ dx,
\]
whenever the right-hand side exists.

\subsection{Notation for the reflection of a hyperplane}\label{rotating}

In the following let $\Omega\subset \R^N$ be an open radial set. For $e\in S^{N-1}$ we set $H_e:=\{x\in \R^N\;:\; x\cdot e>0\}$ and $\Omega_e:=\Omega\cap H_e$. Moreover, we let $\sigma_e:\R^N\to \R^N$, $\sigma_e(x):=x_e:=x-2(x\cdot e) e$ be the reflection at $T_e:=\partial H_e$; for a function $u:\R^N\to \R^m$, $m\in \N$ we let $u_e:=u\circ \sigma_e$ be the \textit{reflected function at~$T_e$}.

\bigskip

For $U\in \ccV_k(\Omega)$ we say that $H_e$ is \textit{dominant}, if $U\geq U_e$ in $H_e$ and we say $H_e$ is \textit{strictly dominant}, if $U\gneq U_e$ in $H_e$. Moreover, we note that 

\begin{lemma}\label{reflection}
Let $\Omega\subset\R^N$ be an open radial set, $e\in S^{N-1}$, and $U\in \ccV_k(\Omega)$. Then
\begin{enumerate}
\item $U_e\in \ccV_k(\Omega)$.
\item If\/ $U\in \ccD_k(\Omega)$ satisfies $U_e=-U$, then  $\chi_{H_e}U\in \ccD_k(\Omega_e)$.
\item Let\/ $U\in \ccV_k(\Omega)$ such that $U_e=-U$ and $U\geq0$ on $H_e\setminus \Omega$. If $\Omega$ is bounded, then $\chi_{H_e}U^-\in \ccD_k(\Omega_e)$ and
\[
\cE_k(U, \chi_{H_e}U^-)\leq -\cE_k(\chi_{H_e}U^-,\chi_{H_e}U^-)\leq -\lambda_1(\Omega_e) \, \|U^-\|_{L^2(\Omega_e)}^2.
\]
\end{enumerate}
\end{lemma}
\begin{proof}
1. follows immediately from the definition of the function space since $k(z_e)=k(z)$ for all $z\in \R^N$, $e\in S^{N-1}$. 2. and 3. follow from \cite[Lemma 3.2]{J16} with \eqref{eigenvalue0} noting that we have
\begin{align*}
\cE_k(U,\, &\chi_{H_e} \, U^-)=\sum_{i=1}^{m}\cE_k(u_i,\, \chi_{H_e}u_i^-)\leq -\sum_{i=1}^{m}\cE_k(\chi_{H_e}u_i^-,\,\chi_{H_e}u_i^-)\\
&\leq -\lambda_1(\Omega) \int_{\Omega}\sum_{i=1}^{m}(u_i^-)^2\ dx =- \lambda_1(\Omega_e) \, \|U^-\|_{L^2(\Omega_e)}^2.
\end{align*}
\end{proof}

\begin{lemma}\label{linearization}
Let $m\in\N$, $\Omega\subset \R^N$ be an open bounded radial set, $F\in C^1([0,\infty)\times \R^m)$, and let $U\in \ccD_k(\Omega)$ with $u_1,\ldots,u_m\in L^{\infty}(\Omega)$ be a solution of
\[
IU=F(|x|,U)\quad\text{ in $\Omega$; }\quad U=0 \quad\text{ in $\R^N\setminus \Omega$.}
\]
Let $e\in S^{N-1}$ and $W:=W_e:=U-U_e$. Then $W_e\in \ccD_k(\Omega)$ is a solution of the linear problem
\begin{equation}\label{eq:linear}
\left\{\quad \begin{aligned}
IW&=C(x) \, W&&\text{in\/ $\Omega$}\\
W&=0 &&\text{in\/ $\R^N\setminus \Omega$}
\end{aligned}\right.
\end{equation}
which satisfies in addition $W=-W_e$. Here, $C(x)=(c_{ij}(x))_{1\leq i,j\leq m}$ where $c_{ij}\in L^{\infty}(\Omega)$, $i,j=1,\ldots,m$ is given by
\[
c_{ij}(x)=\int_0^1\frac{\partial}{\partial u_{j}} \, f_i(|x|, \, U_e+t(U-U_e))\ dt, \quad x\in \Omega.
\]
\end{lemma}
\begin{proof}
Let $e\in S^{N-1}$, $W=(w_1,\ldots,w_m)$ as in the statement and fix $i\in\{1,\ldots,m\}$. Then clearly $W=-W_e$ and by Lemma \ref{reflection} we have $W\in \ccD_k(\Omega)$ and we have in weak sense in $\Omega$
\begin{align*}
Iw_i&=f_i(|x|,U)-f_i(|x|,U_e)=\int_0^1 \frac{\partial}{\partial U} \, f_i(|x|, \, U_e +t (U(x)-U_e(x)))\ dt \cdot (U(x)-U_e(x))\\
&=c_{i1}(x) \, w_1+\ldots +c_{im}(x) \, w_m
\end{align*}
where we have used the mean value theorem.
\end{proof}

\subsubsection{Foliated Schwarz symmetry}\label{foliated}

Denote by $\cH$ the set of open half spaces in $\R^N$. Give $u:\R^N\to \R$, the \textit{polarization $u_H:\R^N\to\R$ of $u$ w.r.t. $H\in \cH$} is given by
\begin{equation}\label{def:polarization}
u_H(x)=\left\{\begin{aligned} & \max\{\, u(x), \, u(\sigma_H(x)) \,\} && x\in H;\\
&\min\{\, u(x), \, u(\sigma_H(x)) \,\} && x\in \R^N\setminus H,\end{aligned}\right.
\end{equation}
where $\sigma_H(x)$ denotes the reflection of $x$ at $\partial H$. For $U=(u_1,\ldots,u_m)\colon\R^N\to \R^m$ and $H\in \cH$ we denote similarly $U_H\colon\R^N\to\R^m$ by $U_H:= ((u_1)_H,\ldots,(u_m)_H)$. Clearly, $U\geq U_e$ if and only if $U=U_{H_e}$. The following Proposition relates the polarization of a function with the property that this function is foliated Schwarz symmetric. 
\begin{prop}[Proposition 3.3, \cite{SW12}]\label{prop-polarization-foliation}
Let\/ $\Omega\subset \R^N$ be an open radial set and let $P$ be a set of functions $u:\R^N\to \R$, which are continuous. Moreover, let
\[
M:=\{e\in S^{N-1}\;:\; u=u_{H_e} \quad \text{on $\Omega_e$ for all $U\in P$ }\}.
\]
Assume that there is $e_0\in M$ such that the following is true:

For all two dimensional subspaces $V\subset \R^N$ with $e_0\in V$ there are $e_+,e_-\in M\cap V$, $e_+\neq e_-$, which are in the same connected component of\/ $M\cap V$ and satisfy $u=u_{e_+}$ and $u=u_{e_-}$ for every $u\in P$.

Then there is $p\in S^{N-1}$ such that for every connected component $D$ on $\Omega$ the functions $u \, \chi_D$ for $u\in P$ are foliated Schwarz symmetric with respect to $p$.
\end{prop}
Proposition \ref{prop-polarization-foliation} is essential in our proofs and we apply it to the family $P=\{u_1,\ldots,u_m\}$, where $(u_1,\ldots,u_m)\in \ccD_k(\Omega)\cap C(\R^N)$ solves \eqref{eq:system-basis}. The assumption of the proposition is verified with the rotating plane method based on the notation of Subsection \ref{rotating}. We note that the polarization of a function in $\cD_k(\Omega)$ remains in $\cD_k(\Omega)$ -- we include a statement of this fact for the reader's convenience in Lemma \ref{reduces-functionals} below.

\section{Linear problems of systems and the maximum principle}\label{mp}

In the following we collect maximum principles needed for our proofs for linear systems of equations. Here, the problems are stated in a half space and the definition of supersolution is adjusted to the oddness of the solution with respect to a hyperplane as presented in Lemma \ref{eq:linear}. 
In the following, let as above $\cH$ be the set of half spaces in $\R^N$ and fix $H\in \cH$ and $D\subset H$, an open bounded set. We denote the reflection at $\partial H$ by $\sigma_H$. Moreover, let $c_{ij}\in L^{\infty}(D)$, $i,j=1,\ldots,m$ be given and denote $C(x):=(c_{ij}(x))_{1\leq i,j\leq m}$. The following maximum principles are for functions $U\in \ccV_k(D)$ such that $U=-U\circ \sigma_H$ and 
\[
\cE_k(U,V)\geq \int_{\Omega}C(x) \, U(x)\cdot V(x)\ dx.
\]
We also say, that $U$ satisfies in weak sense
\begin{equation}\label{prob:linear0}
\left\{\begin{aligned}IU&\geq C(x) \, U(x) &&\text{in $D$}\\
U&\geq 0 && \text{ in $H\setminus D$}\\
U&=U\circ \sigma_H &&\text{in $\R^N$}\end{aligned}\right.
\end{equation}

We call the linear system \eqref{prob:linear0} \textit{weakly coupled (in $D$)}, if
\[
c_{ij}\geq0 \quad\text{ for all $i,j$ such that $i\neq j$.}
\]
Moreover, we call the linear system \eqref{prob:linear0} \textit{fully coupled}, if it is weakly coupled and
\[
\text{for all $i,j$ there is a compact set $K\subset D$ with $|K|>0$ and}\quad \essinf_{K}c_{ij}>0. 
\]

\begin{prop}[Small volume maximum principle for systems]\label{wmp}
Let $c_{\infty}>0$ and $H\in \cH$. Then there is $\delta>0$ such that for any $D\subset H$ open bounded with $|D|<\delta$ the following holds. If $c_{ij}\in L^{\infty}(D)$, $i,j=1,\ldots,m$ are weakly coupled and with $c_{ij}\leq c_{\infty}$ for $i,j=1,\ldots,m$, then any function $U\in \ccV_k(D)$ satisfying \eqref{prob:linear0} satisfies $U\geq 0$ in $D$.
\end{prop}
\begin{proof}
For $m=1$ see \cite[Proposition 3.5]{JW16}. The general case follows similarly. Indeed, let $c_{\infty}>0$ be given and by \eqref{eigenvalue} we may fix $\delta>0$ such that $\lambda_1(D)>2^{m-1}c_{\infty}$ for all $D\subset H$ with $|D|<\delta$. Moreover, by Lemma \ref{reflection} we may choose $W=\chi_{H} \, U^-\in \ccD_k(D)$ as a suitable test function and we have with Lemma \ref{reflection} and the weak coupling assumption
\begin{align*}
-\lambda_1(D) \, &\|U^-\|_{L^2(D)}^2\geq -\cE_k(\chi_H \, U^-, \, \chi_H \, U^-)\geq \cE_k(U, \, \chi_H \, U^-)\geq \int_{D}C(x) \, U\cdot U^-\ dx\\
&= \sum_{i,j=1}^{m}\int_D c_{ij}(x) \, u_i \, u_j^-\ dx=-\sum_{i=1}^{m}\int_D c_{ii}(x) \, (u_i^-)^2\ dx +\sum_{\substack{i,j=1 \\ i\neq j}}^{m}c_{ij}(x) \, (u_i^+-u_i^-)u_j^-\ dx\\
&\geq -c_{\infty}\sum_{i=1}^{m}\int_D (u_i^-)^2\ dx -c_{\infty}\sum_{\substack{i,j=1 \\ i\neq j}}^{m}u_i^-u_j^-\ dx= -c_{\infty}\int_D (\sum_{i=1}^{m}u_i^-)^2\geq -2^{m-1} \, c_{\infty} \, \|U^-\|_{L^2(D)}^2.
\end{align*}
Hence
\[
(2^{m-1} \, c_{\infty}-\lambda_1(D)) \, \|U^-\|_{L^2(D)}^2\geq0,
\]
which is only possible if $U^-=0$ a.e. on $D$.
\end{proof}

\begin{prop}[Strong maximum principle for systems]\label{smp}
Let $H\in \cH$, $D\subset H$ be a domain, and let $c_{ij}\in L^{\infty}(D)$, $i,j=1,\ldots,m$ be strongly coupled. Then for any function $U\in \ccV_k(D)$ satisfying in weak sense \eqref{prob:linear0} with $U\geq 0$ in $H$ we have either $U\equiv 0$ in $D$ or $U>0$ in $D$ in the sense that
\[
\essinf_{K} u_i>0 \quad\text{ for $i=1,\ldots,m$ and all compact $K\subset D$.}
\]
\end{prop}
\begin{proof}
For $m=1$ see \cite[Proposition 3.6]{JW16}. For $m\in \N$ arbitrary, we first note that for any $i=1,\ldots,m$ we have in weak sense
\[
Iu_i\geq \sum_{j=1}^{m} c_{ij}(x) \, u_j\geq c_{ii}(x) \, u_i\quad\text{in $D$,}
\]
so that $u_i\equiv 0$ in $D$ or $u_i>0$ in $D$ by \cite[Proposition 3.6]{JW16}. If $U\not\equiv 0$ in $D$, then there is at least one $i\in \{1,\ldots,m\}$ such that $u_i\not\equiv 0$ in $D$. But then $u_i>0$ in $D$ (in the essential sense). Next, let $j\in\{1,\ldots,m\}$, $j\neq i$ and assume by contradiction that $u_j\equiv 0$ in $D$. Then for $v\in \cD_k(D)$, $v\geq0$ we have
\begin{equation}\label{smp:proof1}
\cE_k(u_j,v)\geq \sum_{k=1}^{m}\int_D c_{kj}(x) \, u_k \, v \ dx\geq \int_D c_{jj}(x) \, u_j \, v+ c_{ij}(x) \, u_i \, v\ dx=\int_D c_{ij}(x) \, u_i \, v\ dx.
\end{equation}
Since there is a compact set $K\subset D$ with $|K|>0$, $\essinf_K\ c_{ij}>0$, and also $\essinf_K\ u_i>0$, and moreover, there is $v\in \cD_k(D)\cap C^2_c(D)$, $v\geq0$ with $v\equiv 1$ on $K$, it follows that by \eqref{smp:proof1} we have
\begin{align*}
0&<\cE_k(u_j,v) = -\int_{D}\int_{\R^N\setminus D} u_j(y) \, v(x) \, k(x-y)\ dy \, dx\\
&=-\int_D \int_{H\setminus D}u_j(y) \, v(x) \, k(x-y)\ dydx -\int_D \int_{H} u_j(\sigma_H(y)) \, v(x) \, k(x-\sigma_H(y))\ dy \, dx\\
&=-\int_D \int_{H\setminus D}v(x) \, u_j(y) \, [k(x-y)-k(x-\sigma_H(y))]\ dy \, dx\leq0,
\end{align*}
where we have used that $u_j(x)=u_j(\sigma_H(x))$ for all $x\in H$ since $U$ solves \eqref{prob:linear0}. Clearly, this however is a contradiction and hence $u_j\equiv 0$ in $D$ is impossible. Thus $u_j>0$ in $D$ by \cite[Proposition 3.6]{JW16} and since $j$ was arbitrary the statement of the Proposition follows.
\end{proof}

\begin{remark}\label{connectedness}
We note that the connectedness of $D$ in Proposition \ref{smp} is not needed, if $k_0$ is strictly decreasing and hence $k>0$ in $\R^N$.
\end{remark}

\begin{remark}\label{nonantisymmetric}
We emphasize that the conclusion of Proposition \ref{wmp} and \ref{smp} also follow if $U\in\ccV_k(D)$ satisfies $IU\geq C(x)U$ in $D$ and $U\geq 0$ on $\R^N\setminus D$. The proof in this case is similar, but simpler. 
\end{remark}

\section{Proof of the symmetry result}\label{sec:proof-main-thm}

Using the notation of the previous sections, Theorem \ref{thm:main1} can be stated equivalently as

\begin{thm}\label{thm:main1-final}
Let $m\in\N$, $\Omega\subset\R^N$ be a bounded radial domain. Assume that $F\in C^1([0,\infty)\times \R^m,\R^m)$, $(r,U)\mapsto F(r,U)$ satisfies
\begin{equation}\label{thm:main1-final-assumption1}
\frac{\partial}{\partial u_i}F_j\geq 0 \quad\text{on $[0,\infty)\times\R^m$, $i,j\in\{1,\ldots,m\}$, $i\neq j$.}
\end{equation}
Let $U\in \ccD_k(\Omega)$ be a bounded continuous solution of \eqref{eq:system2}, and assume that there is $e_0\in S^{N-1}$ such that $U\gneq U_{e_0}=U\circ \sigma_{e_0}$ in $\Omega_{e_0}$ in the sense that $U(x)\geq U_{e_0}(x)$ for all $x\in \Omega_{e_0}$. Suppose, further, that there exist $i\in\{1,\ldots,m\}$ and $x\in \Omega_{e_0}$ such that $u_{i}(x)> u_i(\sigma_{e_0}(x))$.
If either 
\begin{equation}\label{thm:main1-final-assumption1b}
\frac{\partial}{\partial u_i}F_j> 0 \quad\text{on $(0,\infty)\times\R^m$, $i,j\in\{1,\ldots,m\}$, $i\neq j$}
\end{equation}
or $U>0$ in $\Omega$ and
\begin{equation}\label{thm:main1-final-assumption1c}
\frac{\partial}{\partial u_i}F_j> 0\quad\text{on $(0,\infty)^{m+1}$, $i,j\in\{1,\ldots,m\}$, $i\neq j$}
\end{equation}
 then there is $p\in S^{N-1}$ such that $U$ is foliated Schwarz symmetric with respect to $p$ and strictly decreasing in the polar angle.
\end{thm}

\begin{proof}
Denote $W_e:=U-U_e$ for $e\in S^{N-1}$ and note that $W$ satisfies in weak sense
\begin{equation}\label{eq:linearized}
\left\{\ \begin{aligned}
IW_e&= C(x) \, W_e &&\text{in $\Omega$}\\
W_e&=0 &&\text{in $\R^N\setminus \Omega$}\\
W_e&=-W_e\circ \sigma_e &&\text{in $\R^N$,}
\end{aligned}\right.
\end{equation} 
where $C(x)=(c_{ij}(x))_{1\leq i,j\leq m}$ with entries $c_{ij}(x)=\int_0^1\frac{\partial}{\partial u_i}f_j(|x|, \, U_e(x)+t(U(x)-U_e(x)))\ dt$ (see Lemma \ref{linearization}). Note that by our assumptions there is $c_{\infty}>0$ such that
\[
\max_{i,j}\sup_{\substack{x\in \Omega\\ e\in S^{N-1}}}|c_{ij}(x)|\leq c_{\infty}.
\]
By assumption \eqref{thm:main1-final-assumption1} it follows that the system \eqref{eq:linearized} is weakly coupled and, moreover, if \eqref{thm:main1-final-assumption1b} holds or if $U>0$ in $\Omega$ and \eqref{thm:main1-final-assumption1c} holds, then the system is also strongly coupled.

\textit{Step 1:} We claim that
\begin{equation}\label{step1}
W_{e_0}>0 \quad\text{ in $\Omega_{e_0}$.}
\end{equation}
Note that $W_{e_0}\gneq0$ in $\Omega_{e_0}$ and for $W_{e_0}=(w_1,\ldots,w_m)$ we have for $i=1,\ldots,m$
\[
\left\{\ \begin{aligned}
Iw_i&= c_{ii} \, w_i+\sum_{\substack{j=1\\j\neq i}}^mc_{ij} \, w_j &&\text{in $\Omega_{e_0}$}\\
w_i&=0 &&\text{in $H_{e_0}\setminus \Omega_{e_0}$}\\
w_i&=-w_i\circ \sigma_e &&\text{in $\R^N$,}
\end{aligned}\right.
\] 
Hence it by Proposition \ref{smp} that $W_{e_0}>0$ (since $W_{e_0}\equiv0$ is impossible by assumption), that is \eqref{step1} holds.

\bigskip

\textit{Step 2:} Next, by continuity of $U$ and $e\mapsto \sigma_e$, there is for any $\delta>0$ an $\epsilon>0$ such that for any $e\in S^{N-1}$ with $|e-e_0|<\epsilon$ there is $K\subset \Omega_e\cap \Omega_{e_0}$ with
\[
W_e\geq0 \quad \text{in $K$} \quad \text{and}\quad |\Omega_e\setminus K|\leq \delta.
\] 
We claim that there is $\epsilon>0$ such that
\begin{equation}\label{step2}
\text{ $W_e\geq0$ in $\Omega_e$ for $e\in S^{N-1}$ with $|e-e_0|<\epsilon$.}
\end{equation}
To see \eqref{step2}, we use Proposition \ref{wmp}. Fix $\delta>0$ such that $\lambda_1(M)>c_{\infty}$ for any $A\subset\R^N$ with $|A|<\delta$. Let $\epsilon>0$ be given by the above remark and fix $e\in S^{N-1}$ with $|e-e_0|<\epsilon$ and $K\subset \Omega_e\cap \Omega_{e_0}$ with $W_e\gneq0$ in $K$. Finally, let $A:=\Omega_e\setminus K$. As before, let $W_e=(w_1,\ldots,w_m)$ and note that now for $i=1,\ldots,m$
\[
\left\{\ \begin{aligned}
Iw_i&= c_{ii} \, w_i+\sum_{\substack{j=1\\j\neq i}}^mc_{ij} \, w_j &&\text{in $A$}\\
w_i&=0 &&\text{in $H_{e_0}\setminus A$}\\
w_i&=-w_i\circ \sigma_e &&\text{in $\R^N$,}
\end{aligned}\right.
\] 
Since by the assumptions on $F$ (and $U$) this system is weakly coupled, Proposition \ref{wmp} implies $W_e\geq0$ in $H_e$. Hence \eqref{step2} holds.

\bigskip

\textit{Step 3:} Next, we fix $(-\pi,\pi)\ni \phi\to R(\phi)\in O(N)$, such that $R(\phi)$ is a rotation of angle $\phi$ in a fixed direction and put $e^{\phi}:=R(\phi) \, e_0\in S^{N-1}$. Denote
$$
M:=\{\phi\in(-\pi,\pi)\;:\; W_{e^{\phi}}\geq 0 \text{ in $\Omega_{e^\phi}$ }\}.
$$
and let
\[
\phi_+:=\sup M\quad\text{and}\quad \phi_-:=\inf M.
\]
Clearly, $\phi_+\in(\epsilon, \, \pi-\epsilon)$ for some $\epsilon>0$ by Step 1 and similarly $\phi_-\in (-\pi+\epsilon, \, -\epsilon)$ for some $\epsilon>0$. Let $e_+:=e^{\phi_+}$ and $e_-:=e^{\phi_-}$. The proof is finished once we have shown (see Proposition \ref{prop-polarization-foliation})
\begin{equation}\label{step3}
W_{e_+}\equiv 0 \quad\text{in $\Omega_{e_+}\quad$ and}\quad W_{e_-}\equiv 0 \quad\text{in $\Omega_{e_-}$.}
\end{equation}
Note that then there must be $i\in\{1,\ldots,m\}$ such that $w_{e_+,i}\equiv 0$ in $\Omega_{e_+}$, because otherwise a similar argumentation as in Step 1 and Step 2 allows to continue rotating the hyperplanes which is a contradiction to the definition of $e_+$. Similarly, there must be $i\in\{1,\ldots,m\}$ such that $w_{e_-,i}\equiv 0$ in $\Omega_{e_-}$. Let $i\in\{1,\ldots,m\}$ such that $w_{e_+,i}\equiv 0$ in $\Omega_{e_+}$ in $\Omega_{e_+}$. And assume there is $j\in\{1,\ldots,m\}$ such that $w_{e_+,j}\not\equiv0$ in $\Omega_{e_+}$. Then as in Step 1 it follows that this is impossible. Thus $W_{e_+}\equiv 0$ and similarly, also $W_{e_-}\equiv 0$. Hence \eqref{step3} holds.

\bigskip

By Proposition \ref{prop-polarization-foliation} and \eqref{step3} it follows that there is $p\in S^{N-1}$ such that $U$ is foliated Schwarz symmetry, since $e_+$ and $e_-$ are clearly in the same two dimensional component of $e$'s in which $U\geq U\circ \sigma_e$ and $e_+\neq e_-$. The fact that $U$ is strictly decreasing in the polar angle now follows, from Step 2 and with Proposition \ref{smp}, we actually have that $W_{e^{\phi}}>0$ in $\Omega_{e^{\phi}}$ for $\phi\in(\phi_-,\phi_+)$. This finishes the proof. 
\end{proof}

\section{An Application}\label{application}

In the following, we consider the case $m=2$. The system~\eqref{eq:system1} is called \textit{of gradient type} if there exists a scalar function $g(|x|,u_1,u_2)$ such that $f_j(|x|,u_1,u_2) = \frac{\partial g}{\partial u_j}(|x|,u_1,u_2)$ for $j = 1,2$ (see \cite[p.~3]{deFigueiredo}). Let us consider the following system:
\begin{equation}\label{eq:system-gradient}
\left\{\ \begin{aligned}
I u_1&= a_1(x) \, u_1+ |u_2|^q \, |u_1|^{q - 2} \, u_1 &&\text{in $\Omega$}\\
I u_2&= a_2(x) \, u_2+|u_1|^q \, |u_2|^{q - 2} \, u_2  && \text{in $\Omega$}\\
u_1&=u_2=0 &&\text{in $\R^N\setminus \Omega$}
\end{aligned}\right.
\end{equation}
where we assume 
\begin{equation}\label{assumption-k3}
\text{There is $c>0$ and $s\in(0,1)$ such that $k_0(r)\geq cr^{-1-2s}$ for $r\in(0,1)$,}
\end{equation}
$\Omega$ is an open bounded set in $\R^N$, $N\geq 2$ with Lipschitz boundary, and
\begin{equation}
\text{$a_1,a_2\in L^{\infty}(\R)$ with $a_1\neq a_2$ and $\|a^+_i\|_{L^{\infty}(\R)}<\lambda_1(\Omega)$ for $i=1,2$,}
\end{equation}
where $\lambda_1(\Omega)>0$ is the first eigenvalue of $I$ (see \eqref{first-eigenvalue}, \eqref{eigenvalue0}).
Moreover, we let $1 < q < \frac{N}{\, N - 2s \,}$ (cf.~\cite[(2.2)]{MMP}). Note that we clearly have $N>2s$ since $s\in(0,1)$ and that the kernel of the fractional Laplacian $(-\Delta)^s$ satisfies \eqref{assumption-k3}. Moreover, the system~\eqref{eq:system-gradient} is of gradient type, with $g$ given by $g(u_1,u_2) = \frac1q \, |u_1 \, u_2|^q$. In the following let $f_j(u_1,u_2) = |u_{3 - j}|^q \, |u_j|^{q - 2} \, u_j$ for $j = 1,2$, then we see immediately that
$$
\frac\partial{\, \partial u_i \,} \, f_j
=
q \, |u_i \, u_j|^{q - 2} \, u_i \, u_j
\quad
\mbox{for $i = 1,2$ and $j = 3 - i$.}
$$
Hence, the system~\eqref{eq:system-gradient} is \textit{weakly coupled} as long as the product $u_1 \, u_2$ is non-negative in~$\Omega$. A  similar system is considered in \cite{MMP} (see also \cite[(4.1)]{DP13}) with a local operator in place of~$I$, and with the bounded set~$\Omega$ replaced by the whole space~$\mathbb R^N$. An existence proof of a pair of non-negative, radially symmetric solutions $u_1,u_2 \ge 0$ satisfying $u_1 + u_2 \not \equiv 0$ in~$\mathbb R^N$ is given there. In the present paper, to keep the argument as transparent as possible, the nonlinearities in~\eqref{eq:system-gradient} are simpler than those in~\cite{MMP}. However, the parameter $\omega$ occurring there is replaced by the function~$a_2(x)$. In order to prove Theorem \ref{thm:main3}, we begin with an existence statement.

\begin{thm}[Existence of non-trivial solutions]\label{thm:gradient}
Let $\Omega\subset\R^N$, $N\geq 2$ be a bounded open set with Lipschitz boundary, assume $k_0$ satisfies \eqref{assumption-k3} for some $s\in(0,1)$, and\/ $1 < q < \frac{N}{\, N - 2s \,}$ . Then system \eqref{eq:system-gradient} has a weak solution $(u,v)$ satisfying $u,v \not \equiv 0$ in~$\Omega$ and $u \not \equiv v$.
\end{thm}

 The existence proof is based on the mountain-pass theorem (see, for instance, \cite[Theorem~8.2]{AM} or \cite[Chapter~III, Theorem~6.1, p.~109]{Struwe}). More precisely, we consider the functional
\begin{equation}\label{functional-gradient}
J(u,v)
=
\textstyle \frac12 \, \Vert (u,v) \Vert^2
-
\frac1q \, \Vert uv \Vert^q_{L^q(\Omega)}
,
\end{equation}
where we have used the notation $\Vert (u,v) \Vert^2 = \cE_k(u,u) - \int_{\Omega}a_1 \, u^2+a_2 \, v^2 \ dx
+ \cE_k(v,v)$, for shortness. Note that it is easy to see that $\cD_k(\Omega)$ is continuously embedded into $\cH^s_0(\Omega)=\{u\in H^s(\R^N)\;:\; u=0\ \text{on $\R^N\setminus \Omega$}\}$ and hence by the assumption $q < \frac{N}{\, N - 2s \,}$, it follows by the Sobolev embedding that $\cD_k(\Omega)$ is compactly embedded into $L^{2q}(\Omega)$ (see \cite[Theorem~6.7]{DPV}). Hence the product $uv$ belongs to $L^q(\Omega)$, and the functional $J(u,v)$ is well defined on $\ccD_k(\Omega)$. The differential $J'$ at $(u,v)$ is the linear operator~$L$ given by
\begin{align}
\nonumber
L(\varphi,\psi)
&=
\cE_k(u,\varphi)
-
\int_\Omega \big( a_1 \, u \, \varphi+ a_2 \, v \, \psi \big) \, dx
+
\cE_k(v,\psi)
\\
\noalign{\medskip}
\label{L}
&
-
\int_\Omega
|v|^q \, |u|^{q - 2} \, u \, \varphi \, dx
-
\int_\Omega
|u|^q \, |v|^{q - 2} \, v \, \psi \, dx
,
\end{align}
where $(\varphi,\psi)$ ranges in $\ccD_k(\Omega)$. Hence the critical points of~$J$ are the weak solutions of~\eqref{eq:system-gradient}. To apply the mountain pass theorem, we collect in the next section several properties of the Nehari Manifold $\cN$. For general applications of the mountain pass theorem to nonlocal operators, see also \cite{SV12,SV13}.

\subsection{The Nehari manifold}

 In the sequel we refer to the Nehari manifold $\cN$ associated to the functional~$J$. First define the functional
$$
G(u,v)
=
\textstyle \frac12 \, \Vert (u,v) \Vert^2
-
\Vert uv \Vert^q_{L^q(\Omega)}
,
$$
and then let
\begin{align}
\nonumber
\cN
&=
\{\, (u,v) \in \ccD_k(\Omega) \setminus (0,0)
\mid
\mbox{the differential $J'(u,v)$ vanishes in the direction of $(u,v)$}
\,\}
\\
\noalign{\medskip}
\label{Nehari}
&=
\{\, (u,v) \in \ccD_k(\Omega) \setminus (0,0)
\mid
G(u,v) = 0
\,\}
,
\end{align}
where the last equality is readily obtained by letting $(\varphi,\psi) = (u,v)$ in~\eqref{L}. In order to prove the existence of polarized solutions of system~\eqref{eq:system-gradient}, we need

\begin{lemma}\label{lemma:Nehari}
~\begin{enumerate}
\item The Nehari manifold $\cN$ is a $C^1$-manifold of codimension one in\/ $\ccD_k(\Omega)$.
\item If\/ $(u,v)$ belongs to~$\cN$, then the direction of\/ $(u,v)$ is non-tangential to~$\cN$.
\item The manifold $\cN$ keeps far from the origin in the sense that there exists $r_0 > 0$ such that if\/ $\Vert (u,v) \Vert < r_0$ then $(u,v) \not \in \cN$.
\end{enumerate}
\end{lemma}

\begin{proof}
Choose a point $(u_0,v_0) \in \cN_+$, and observe that the product $u_0 \, v_0$ cannot vanish identically (that would be in contrast with~\eqref{Nehari}). In a neighborhood of $(u_0,v_0)$, the Nehari manifold is the set of zeros of the functional $G(u,v)$, whose differential is the linear functional $G'(u,v)$ given by
\begin{align*}
G'(u,v) \colon (\varphi,\psi)\
\mapsto\
\cE_k(u,\varphi)
&-
\int_\Omega \big( a_1 \, u \, \varphi +a_2 \, v \, \psi \big) \, dx
+
\cE_k(v,\psi)
\\
\noalign{\medskip}
&
-
q \int_\Omega
|v|^q \, |u|^{q - 2} \, u \, \varphi \, dx
-
q \int_\Omega
|u|^q \, |v|^{q - 2} \, v \, \psi \, dx
.
\end{align*}
To prove Claim~1 we show that the image of $(\varphi,\psi)$ through $G'(u,v)$ does not vanish for every $(\varphi,\psi)$. This is achieved by letting $(\varphi,\psi) = (u,v)$ and taking into account that $G(u,v) = 0$, which yields
$$
G'(u,v) \colon (u,v) \mapsto
\Vert(u,v)\Vert^2 - 2q \, \Vert uv \Vert_{L^q(\Omega)}^q
=
(1 - q) \, \Vert(u,v)\Vert^2 < 0
.
$$
This implies that $\cN$ is a $C^1$-manifold of codimension~$1$, and the direction of~$(u,v)$ is non-tangential, thus proving Claims 1 and~2 at once. To prove the last claim, observe that by the Poincar\'e inequality and the Sobolev embedding we have
\begin{align*}
\Vert (u,v) \Vert^2&\geq \cE_k(u,u)-\|a_1^+\|_{L^{\infty}(\R)} \, \|u\|_{L^2(\R^)}^2+\cE_k(v,v)-\|a_2^+\|_{L^{\infty}(\R)} \, \|v\|_{L^2(\R^)}^2
\\
\noalign{\medskip}
&\ge
C_0 \, \big(\cE_k(u,u)+\cE_k(v,v)\big)
\ge
C
\,
\big(
\Vert u \Vert^2_{L^q(\Omega)}
+
\Vert v \Vert^2_{L^q(\Omega)}
\big)
\\
\noalign{\medskip}
&
\ge
C_1
\,
\Vert uv \Vert_{L^q(\Omega)}
,
\end{align*}
where $C_0,C,C_1 > 0$ are constants. Hence we may write $\Vert uv \Vert^q_{L^q(\Omega)} \le C_2 \, \Vert (u,v) \Vert^{2q}$, and therefore the inequality:
\begin{equation}
G(u,v)
\ge
\Vert (u,v) \Vert^2
\,
\Big(
\textstyle \frac1{\, 2 \,}
-
C_2
\,
\Vert (u,v) \Vert^{2(q - 1)}
\Big)
\ge
\textstyle \frac1{\, 3 \,}
\,
\Vert (u,v) \Vert^2
\label{positivity}
\end{equation}
holds provided that $\Vert (u,v) \Vert < r_0$ with a conveniently small~$r_0 > 0$. The last claim follows, and the proof is complete.
\end{proof}

\begin{proof}[Proof of Theorem \ref{thm:gradient}]
Let us check that the functional~\eqref{functional-gradient} satisfies the assumptions of the mountain-pass theorem.

\bigskip

\textit{Step 1:} The equality $J(0,0) = 0$ holds, and there exists $r_0 > 0$ such that $J(0,0) > 0$ for all $u,v \in \cD_k(\Omega)$ satisfying $0 < \Vert (u,v) \Vert < r_0$. Indeed, arguing as in the proof of the last claim of Lemma~\ref{lemma:Nehari}, and writing $J$ in place of~$G$, we arrive at $J(u,v) \ge \frac1{\, 3 \,} \, \Vert (u,v) \Vert^2$ for $\Vert (u,v) \Vert < r_0$ (cf.~\eqref{positivity}). The same inequality also shows that $J(u,v) \ge r_0^2/3$ whenever $\Vert (u,v) \Vert = r_0$.

\bigskip

\textit{Step 2:} The functional $J$ is unbounded from below. To see this, fix a pair $(u,v) \in \ccD_k(\Omega)$ satisfying $\Vert uv \Vert_{L^q(\Omega)}^q > 0$ in~$\Omega$. Since for every $t \ge 0$ we have
\begin{equation}\label{unbounded}
J(tu, tv)
=
t^2
\,
\Big(
\textstyle \frac1{\, 2 \,}
\,
\Vert (u,v) \Vert^2
-
\frac{\, t^{2(q - 1)} \,}q
\,
\Vert uv \Vert^q_{L^q(\Omega)}
\Big)
,
\end{equation}
we see that $J(tu,tv) \to -\infty$ as $t \to \infty$, hence $J$ is unbounded from below, as claimed.

\bigskip\goodbreak

\textit{Step 3:} The last condition needed to apply the mountain-pass theorem is the Palais-Smale compactness condition. More precisely, assume that a sequence of pairs $(u_i,v_i) \in \ccD_k(\Omega)$ satisfies $J(u_i, v_i) \to c \in (0,\infty)$ as $i\to\infty$ in the Euclidean topology of the real line, as well as $J'(u_i,v_i) \to 0$ as $i\to\infty$ in the strong topology of the dual space $(\ccD_k(\Omega))'$. Then we have to prove the existence of a strongly convergent subsequence in $\ccD_k(\Omega)$. To this purpose, observe that the differential $J'$ at the point $(u_i,v_i)$ is the linear functional $L_i(\varphi,\psi)$ given by
\begin{align}
\label{differential}
L_i(\varphi,\psi)
=
\cE_k(u_i,\varphi)
&-
\int_\Omega \big( a_1 \, u_i \, \varphi +a_2 \, v_i\, \psi \big) \, dx
+
\cE_k(v_i,\psi)
\\
\noalign{\medskip}
&
-
\int_\Omega
|v_i|^q \, |u_i|^{q - 2} \, u_i \, \varphi \, dx
-
\int_\Omega
|u_i|^q \, |v_i|^{q - 2} \, v_i \, \psi \, dx
,
\end{align}
where $(\varphi,\psi)$ ranges in $\ccD_k(\Omega) \subset (L^{2q}(\Omega))^2$. In the special case when $(\varphi,\psi) = (u_i,v_i)$ we find $L_i(\phi,\psi) = \Vert (u_i,v_i) \Vert^2 - 2 \, \Vert uv \Vert^q_{L^q(\Omega)}$, and hence
\begin{equation}\label{expansion1}
2 \, \Vert uv \Vert^q_{L^q(\Omega)}
=
\Vert (u_i,v_i) \Vert^2
-
L_i(u_i,v_i)
.
\end{equation}
Let us combine the equality above with the assumption that $J'(u_i,v_i) \to 0$ as $i\to\infty$ strongly. Such an assumption implies $L_i(u_i,v_i) = o(1) \, \Vert (u_i,v_i) \Vert$ as $i\to\infty$: by plugging this into~\eqref{expansion1} we obtain
\begin{equation}\label{expansion2}
\Vert u_i \, v_i \Vert^q_{L^q(\Omega)}
=
\textstyle \frac1{\, 2 \,}
\,
\Vert (u_i,v_i) \Vert^2
+
o(1)
\,
\Vert (u_i,v_i) \Vert
\quad\text{as $i\to\infty$.}
\end{equation}
Now we are ready to prove the existence of a strongly convergent subsequence. As usual, the proof is divided into two parts.

\bigskip

\textit{Part $i$:} The sequence $(u_i,v_i)$ is bounded. Indeed, if we assume $\Vert (u_i,v_i) \Vert \to \allowbreak \infty$ for $i\to\infty$, then we reach a contradiction by the following argument. Taking~\eqref{expansion2} into account, we have
$$
J(u_i,v_i)
=
\textstyle \frac1{\, 2 \,}
\,
(1 - \frac1{q})
\,
\Vert (u_i,v_i) \Vert^2
+
o(1)
\,
\Vert (u_i,v_i) \Vert
\to \infty
,
$$
which contradicts the assumption $J(u_i,v_i) \to c < \infty$ for $i\to\infty$. Hence the sequence $(u_i,v_i)$ must be bounded, as claimed.

\bigskip\goodbreak

\textit{Part $ii$:} Once we know that the sequence $(u_i,v_i)$ is bounded in $\ccD_k(\Omega)$, the proof of the existence of a strongly converging subsequence is standard: see \cite[p.~125]{AM} and \cite[Proposition~2.2]{Struwe}. To be more precise, by the weak compactness theorem in Hilbert spaces there exists a subsequence, still denoted by $(u_i,v_i)$, weakly convergent to some $(u,v) \in \ccD_k(\Omega)$. Furthermore, since $q < \frac{N}{\, N - 2s \,}$, the set $\cD_k(\Omega)\subset \cH^s_0(\Omega)$ is compactly embedded in the Lebesgue space~$L^{2q}(\Omega)$, hence we may assume that when $i \to \infty$ the sequences $(u_i),(v_i)$ converge to $u,v$, respectively, strongly in $L^{2q}(\Omega)$, and therefore $\Vert u_i \, v_i \Vert^q_{L^q(\Omega)} \to \Vert uv \Vert^q_{L^q(\Omega)}$. This and~\eqref{expansion2}, taking the boundedness of the sequence $(u_i,v_i)$ into account, imply
\begin{equation}\label{limit}
\lim_{i \to \infty}
\Vert (u_i,v_i) \Vert^2
=
2 \, \Vert uv \Vert^q_{L^q(\Omega)}
.
\end{equation}
Consider the functional $L(\varphi,\psi)$ in~\eqref{L}. Taking~\eqref{differential} into account, and since $(u_i,v_i) \rightharpoonup (u,v)$ weakly in $\ccD_k(\Omega)$, and $(u_i,v_i) \to (u,v)$ strongly in $(L^{2q}(\Omega))^2$, we deduce
$$
\lim_{i \to \infty}
(L_i - L)(\varphi,\psi)
=
0
$$
for every $(\varphi,\psi) \in \ccD_k(\Omega)$: thus, we have proved the weak-$*$ convergence $L_i \stackrel{*}{\rightharpoonup} L$. But since $L_i \to 0$ strongly by assumption, we must have $L = 0$. In particular, $(u,v) \in \cN$. By comparing~\eqref{Nehari} with~\eqref{limit} we deduce
$$
\lim_{i \to \infty}
\Vert (u_i,v_i) \Vert
=
\Vert (u,v) \Vert
.
$$
Finally, by recalling that the weak convergence in a Hilbert space together with the convergence of the norms to the norm of the limiting function implies the strong convergence, we conclude that $(u_i,v_i) \to (u,v)$ strongly in $\ccD_k(\Omega)$, which completes the proof of the Palais-Smale compactness condition.

\bigskip

At this point the mountain-pass theorem implies the existence of a critical point $(u,v) \ne (0,0)$ of the functional~$J$, which is therefore a weak solution $(u_1,u_2) = (u,v)$ of the system~\eqref{eq:system-gradient}. By the mountain-pass theorem we also know that the two identities $u_1 \equiv 0$ and $u_2 \equiv 0$ cannot hold at once, but we may, in principle, have $u_2 \equiv 0$. However, if $u_2$ vanishes identically, then system~\eqref{eq:system-gradient} implies $I u_1 = a_1 \, u_1$ in~$\Omega$, $u_1 = 0$ in $\mathbb R^N \setminus \Omega$, hence $u_1$ should also vanish identically by unique solvability and the maximum principle, a contradiction. A similar argument shows that $u_1 \not \equiv 0$, hence $u_1,u_2 \not \equiv 0$ in~$\Omega$. Finally, if $u_1 \equiv u_2$ in~\eqref{eq:system-gradient}, then by comparing the two equations ---recall $a_1\neq a_2$--- we obtain $u_1 \equiv 0$, which has been just excluded. Hence $u_1,u_2$ are distinct functions, and the proof is complete.
\end{proof}

\subsection{Positivity}

Let us now turn to show that the solutions $u,v$ obtained so far do not change sign. To this aim we need to define the set of paths $\Gamma = \{\, \gamma \in C^0([0,1], \, \ccD_k(\Omega)) \mid \gamma(0) = 0,\ J(\gamma(1)) < 0 \,\}$ and the two infima
$$
c
=
\inf_{\gamma \in \Gamma}
\max_{t \in [0,1]}
J(\gamma(t)),
\qquad
\qquad
c_\cN
=
\inf_{(u,v) \in \cN}
J(u,v)
.
$$
\begin{lemma}\label{coincide}
The two values $c,c_\cN$ defined above are positive and coincide.
\end{lemma}

\begin{proof}
The argument is similar to~\cite[Theorem~4.2]{Willem} (for a scalar equation) and \cite[Lemma~3.2]{MMP} for a system of local equations. Let us verify that $c \le c_\cN$. Take $(u,v) \in \cN$ and observe that $\Vert uv \Vert_{L^q(\Omega)}^q > 0$, otherwise we would reach a contradiction with~\eqref{Nehari}. Then \eqref{unbounded} applies, and the path $\gamma(t) = (tu,tv)$, $t \in [0,\infty)$, starts from the origin and satisfies $\lim\limits_{t \to \infty} J(\gamma(t)) = -\infty$. Of course, we may find a reparametrization such that $J(\gamma(1)) < 0$, but we prefer to avoid unnecessary technicalities.
\begin{figure}[ht]
\centering
\includegraphics{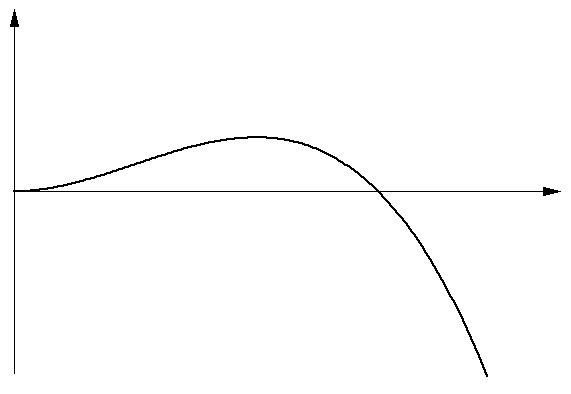}
\caption{The function $f(t)$}
\label{fig:1}
\end{figure}
Taking~\eqref{Nehari} into account, a straightforward computation shows that the real-valued function $f(t) = J(\gamma(t))$ of the real variable~$t > 0$ (whose graph is outlined in Figure~\ref{fig:1}) satisfies $f'(1) = 0$. Furthermore, $f$ attains its maximum (which is positive) at $t = 1$ and hence $c \le \max\limits_{t \ge 0} J(\gamma(t)) = J(u,v)$. Since $(u,v) \in \cN$ is arbitrary, we may write $c \le c_\cN$. To prove the converse, recall that by Theorem~\ref{thm:gradient} there exists $(u,v) \ne (0,0)$ such that $J'(u,v) = 0$ and $J(u,v) = c > 0$ (this is a by-product of the mountain-pass theorem). But then $(u,v) \in \cN$ and therefore $c_\cN \le J(u,v) = c$. The lemma follows.
\end{proof}

\begin{prop}\label{sign}
The two functions $u,v$ obtained by Theorem~\ref{thm:gradient} do not change sign.
\end{prop}

\begin{proof}
The argument is based on the combination of three inequalities:
\begin{enumerate}
\item Since $(u,v)$ is a critical point of the functional~$J$, we have $G(u,v) = 0$ (see~\eqref{Nehari}), hence the function $f(t) = J(tu, \allowbreak \, tv) = \frac{\, t^2 \,}{\, 2 \,} \, \Vert (u,v) \Vert^2 - \frac{\, t^{2q} \,}{\, q \,} \, \Vert uv \Vert_{L^q(\Omega)}^q$ satisfies $f'(1) = 0$. An elementary computation shows that
$$
f(t) \le f(1) = J(u,v)
$$
for all $t > 0$, with equality if and only if $t = 1$ (the graph of $f$ is outlined in Figure~\ref{fig:1}).
\item Since the graph of $g(t) = G(t|u|, \, t|v|) = \frac{\, t^2 \,}{\, 2 \,} \, \Vert (|u|, \allowbreak |v|) \Vert^2 - t^{2q} \, \Vert uv \Vert_{L^q(\Omega)}^q$ has the same shape as the one of~$f$, there exists $t_0 > 0$ such that $g(t_0) = 0$. Then, by~\eqref{Nehari} we have $(t_0 |u|, \allowbreak \, t_0 |v|) \in \cN$, and by Lemma~\ref{coincide} we get
$$
J(u,v) = c \le J(t_0 |u|, \allowbreak \, t_0 |v|)
.
$$
\item Using Lemma \ref{vk-properties} we obtain $J(t_0 |u|, \, t_0 |v|) \le J(t_0\, u, \, t_0 v) = f(t_0)$.
\end{enumerate}
\goodbreak
In conclusion, we arrive at $J(t_0 |u|, \, t_0 |v|) = f(t_0) = J(u,v)$, whence we deduce that $t_0 = 1$ and $J(|u|,|v|) = J(u,v)$. This and Lemma \ref{vk-properties} imply that either $u^+$ or $u^-$ vanishes almost everywhere, and either $v^+$ or $v^-$ vanishes almost everywhere. The claim follows.
\end{proof}

\begin{cor}[Existence of positive solutions]\label{positive}
Let $\Omega\subset\R^N$, $N\geq 2$ be a bounded open set with Lipschitz boundary, assume $k_0$ satisfies \eqref{assumption-k3} for some $s\in(0,1)$, and $1 < q < \frac{N}{\, N - 2s \,}$ . Then system \eqref{eq:system-gradient} has a weak solution $(u_1,u_2)$ satisfying $u_1,u_2> 0$ in~$\Omega$ and $u_1 \not \equiv u_2$.
\end{cor}

\begin{proof}
Consider the non-negative functions $u_1 = |u|$ and $u_2 = |v|$, where $(u,v)$ is the weak solution whose existence follows from Theorem~\ref{thm:gradient}. In view of Proposition~\ref{sign}, we must have either $u_1 = u$ or $u_1 = -u$, and either $u_2 = v$ or $u_2 = -v$. Therefore the pair $(u_1,u_2)$ satisfies~\eqref{eq:system-gradient}. But then $(u_1,u_2)$ also satisfies the system of uncoupled inequalities
$$
\left\{\ \begin{aligned}
Iu_1&\geq a_1(x) \, u_1 &&\text{in $\Omega$}\\
Iu_2&\ge a_2(x) \, u_2 &&\text{in $\Omega$}\\
u_1=u_2&=0 &&\text{in $\R^N\setminus \Omega$}
\end{aligned}\right.
$$
By the strong maximum principle (see Proposition \ref{smp} and Remark \ref{nonantisymmetric}) we have that for each $j = 1,2$ either $u_j > 0$ in~$\Omega$ or $u_j \equiv 0$ in $\mathbb R^N$, and the conclusion follows from Theorem~\ref{thm:gradient}.
\end{proof}

\begin{remark}
$(i)$~Since the weak solution $(u_1,u_2)$ whose existence is asserted by Corollary~\ref{positive} minimizes the functional~$J$ over the Nehari manifold~$\cN \!$, we say that $(u_1,u_2)$ is a \emph{ground state}.\\
$(ii)$~The pair $(-u_1, u_2)$ is also a weak solution, as well as $(u_1, -u_2)$ and $(-u_1,-u_2)$: the assertion follows by replacing $u_j$ in~\eqref{eq:system-gradient} with $\pm u_j$, $j = 1,2$.
\end{remark}

\subsection{Polarized solutions}

The main result in this paragraph states that if $\Omega$ is symmetric, then system~\eqref{eq:system-gradient} admits a solution made up of two polarized functions. Before proceeding further, observe that in our notation we may write $u_{\sigma_H(H)}(x) = u_H(\sigma_H(x))$. Let us describe the effect of polarization on the functionals $J$ and~$G$:

\begin{lemma}[Functionals reduced by polarization]\label{reduces-functionals}
Let\/ $\Omega$ be a bounded open set in\/~$\mathbb R^N$, symmetric with respect to the hyperplane $\partial H$ for some half-space~$H$. Moreover, assume that $a_i$ is symmetric with respect to $\partial H$ for $i=1,2$. For every $(u,v) \in \ccD_k(\Omega)$ satisfying $u,v \ge 0$ in\/~$\mathbb R^N$ we have
\begin{equation}\label{reduces}
J(u_H,v_H) \le J(u,v)
.
\end{equation}
Furthermore, if $k_0$ is strictly decreasing and $J(u_H,v_H) = J(u,v)$ then:
\begin{enumerate}
\item either $u = u_H$ or $u = u_{\sigma_H(H)}$;
\item either $v = v_H$ or $v = v_{\sigma_H(H)}$;
\item inequality~\eqref{product:non-negative} is satisfied.
\end{enumerate}
The lemma also holds with $G$ in place of\/~$J$.
\end{lemma}

\goodbreak

\begin{proof}
By Proposition~\ref{reduces-energy} we have $\cE_k(u_H,u_H) \le \cE_k(u,u)$ and $\cE_k(v_H,v_H) \le \cE_k(v,v)$. Furthermore, by Cavalieri principle and since due to the symmetry of $a_1,a_2$ we also have $ \int_{\Omega} a_1 \, u^2\ dx= \int_{\Omega} (a_1)_H \, (u_H)^2\ dx$ and  $ \int_{\Omega} a_2 \, v^2\ dx= \int_{\Omega} (a_2)_H \, (v_H)^2\ dx$. This and Proposition~\ref{prop:increases} prove~\eqref{reduces}. Now suppose that \eqref{reduces} holds with equality. We may write
\begin{align*}
0 = J(u,v) - J(u_H,v_H)
&=
\big( \cE_k(u,u) - \cE_k(u_H,u_H) \big)
+
\big( \cE_k(v,v) - \cE_k(v_H,v_H) \big)
\\
\noalign{\medskip}
&-
\textstyle\frac1{\, q \,}
\big(
\Vert uv \Vert_{L^q(\Omega)}^q
-
\Vert u_H \, v_H \Vert_{L^q(\Omega)}^q
\big)
\end{align*}
hence
$$
\big( \cE_k(u,u) - \cE_k(u_H,u_H) \big)
+
\big( \cE_k(v,v) - \cE_k(v_H,v_H) \big)
=
\textstyle\frac1{\, q \,}
\big(
\Vert uv \Vert_{L^q(\Omega)}^q
-
\Vert u_H \, v_H \Vert_{L^q(\Omega)}^q
\big)
\le
0
,
$$
where the last inequality follows from Proposition~\ref{prop:increases}. Since the right-hand side cannot be negative by Proposition~\ref{reduces-energy}, it must vanish. But then Proposition~\ref{reduces-energy} implies that either $u = u_H$ or $u = u_{\sigma_H(H)}$, and either $v = v_H$ or $v = v_{\sigma_H(H)}$, as claimed, and  Proposition~\ref{prop:increases} implies that \eqref{product:non-negative} holds. The argument obviously applies to the functional~$G$ as well.
\end{proof}

\begin{remark}
When  $u_H = u \ne u_{\sigma_H(H)}$ and $v_H \ne v = v_{\sigma_H(H)}$, and the product in~\eqref{product:non-negative} vanishes almost everywhere in~$\Omega$, the equality holds in\/~\eqref{reduces} although $u$ and~$v$ have opposite polarizations.
\end{remark}

\begin{thm}[Solutions are polarized]\label{thm:polarization}
Let\/ $\Omega\subset\R^N$, $N\geq 2$ be a bounded open set with Lipschitz boundary and assume $\Omega$ is symmetric with respect to some hyperplane $\partial H$ of a half-space $H$. Assume further that $k_0$ satisfies \eqref{assumption-k3} for some $s\in(0,1)$, and $1 < q < \frac{N}{\, N - 2s \,}$. Then system \eqref{eq:system-gradient} has a weak solution $(u_1,u_2)$ satisfying $u_1,u_2>0$ in $\Omega$, which satisfy either $u_j = (u_j)_H$ for both $j = 1,2$, or  $u_j = (u_j)_{\sigma_H(H)}$ for both $j = 1,2$. Furthermore, if $u_j$ is symmetric with respect to~$\partial H$ for some $j \in \{\, 1,2 \,\}$, then $u_{3 - j}$ is also symmetric.
\end{thm}

\begin{proof}
\textit{Step 1:} Construction of a polarized solution. Denote by $v_j = (u_j)_H$ the polarization of $u_j$ for $j = 1,2$. By Lemma~\ref{reduces-functionals}, we find $G(v_1,v_2) \le G(u_1,u_2) = 0$, hence the real-valued function $g(t) = \frac{\, t^2 \,}{\, 2 \,} \, \Vert (v_1, \allowbreak v_2) \Vert^2 \allowbreak - t^{2q} \, \Vert v_1 \, v_2 \Vert_{L^q(\Omega)}^q$, whose graph has the shape depicted in Figure~\ref{fig:1}, satisfies $g(1) = G(v_1,v_2) \le 0$. Consequently, there exists $t_0 \in (0,1]$ such that $(t_0 \, v_1, \, t_0 \, v_2) \in \cN$. We may write $J(t_0 \, v_1, \, t_0 \, v_2) \le J(t_0 \, u_1, \, t_0 \, u_2)$ by Lemma~\ref{reduces-functionals}, and $J(t_0 \, u_1, \, t_0 \, u_2) \le J(u_1,u_2)$ because the function $f(t) = J(tu_1, \allowbreak tu_2)$ attains its maximum at $t = 1$, hence
\begin{equation}\label{smaller}
J(t_0 \, v_1, \, t_0 \, v_2)
\le
J(t_0 \, u_1, \, t_0 \, u_2)
\le
J(u_1,u_2)
.
\end{equation}
Furthermore, recall that the value $c = J(u_1,u_2)$ is the minimum of~$J$ constrained to~$\cN$ by Lemma~\ref{coincide}: this and~\eqref{smaller} imply $t_0 = 1$ and $J(u_1,u_2) = J(v_1,v_2)$. Hence the pair $(v_1,v_2)$, which is made up of polarized functions, positive in~$\Omega$, is also a minimizer of the functional~$J$ constrained to~$\cN$, and therefore the intrinsic gradient, also called the tangential gradient, of the functional~$J$ on the manifold~$\cN$ vanishes there.

\bigskip

Let us prove that the normal component of the gradient vanishes as well. By Lemma~\ref{lemma:Nehari} we know that the direction of $(u,v)$ is non-tangential to~$\cN$. Furthermore, by~\eqref{Nehari}, the differential $J'$ at any $(u,v) \in \cN$ vanishes in the direction of~$(u,v)$, hence the normal component of the gradient also vanishes, as claimed. But then $J'(v_1,v_2) = 0$, and therefore the pair $(v_1,v_2)$ is a weak solution of system~\eqref{eq:system-gradient}.

\bigskip

\textit{Step 2:} Comparison between solutions. Since $J(u_1,u_2) = J(v_1,v_2)$, by Lemma~\ref{reduces-functionals} we have that either $u_1 = v_1$ or $u_1 = (u_1)_{\sigma_H(H)}$, and either $u_2 = v_2$ or $u_2 = (u_2)_{\sigma_H(H)}$. To prove the theorem we have to exclude two cases: the case when $u_1 \ne v_1$ and $u_2 = v_2$, and the case when $u_1 = v_1$ and $u_2 \ne v_2$. We examine the first case in detail, the second one being analogous. Suppose, by contradiction, that $u_1 \ne v_1$ and $u_2 = v_2$. Then $u_1 = (u_1)_{\sigma_H(H)}$ and there exists a set $X_1 \subset \Omega \cap H$ having positive measure and such that $u_1(x) < u_1(\sigma_H(x)) = v_1(x)$ for every $x \in X_1$. We may assume that $u_1(x) = u_1(\sigma_H(x)) = v_1(x)$ in $(\Omega \cap H) \setminus X_1$. Recall that the pairs $(u_1,u_2)$ and $(v_1,v_2) = (v_1,u_2)$ are both critical points of the functional~$J$. Now the condition $J'(u_1,u_2) = J'(v_1,u_2) = 0$ comes into play: we have
$$
\cE_k(u_1,\varphi)
-
\int_\Omega \big( a_1 \, u_1 \, \varphi + a_2 \, u_2 \, \psi \big) \, dx
+
\cE_k(u_2,\psi)
=
\int_\Omega
u_2^q \, u_1^{q - 1} \, \varphi \, dx
+
\int_\Omega
u_1^q \, u_2^{q - 1} \, \psi \, dx
$$
for every $(\varphi,\psi) \in \ccD_k(\Omega)$ (cf.~\eqref{differential}), and similarly
$$
\cE_k(v_1,\varphi)
-
\int_\Omega \big( a_1 \, v_1 \, \varphi + a_2 \, u_2 \, \psi \big) \, dx
+
\cE_k(u_2,\psi)
=
\int_\Omega
u_2^q \, v_1^{\, q - 1} \, \varphi \, dx
+
\int_\Omega
v_1^{\, q} \, u_2^{q - 1} \, \psi \, dx
.
$$
Letting $\varphi = 0$ and $\psi  = (u_1  - v_1 )^-$, and subtracting the second equality from the first one, we obtain
$$
0
=
\int_\Omega
(u_1^q - v_1^q) \, u_2^{q - 1} \, \psi \, dx
=
-
\int_{X_1}
(u_1^q - v_1^q) \, u_2^{q - 1} \, (u_1 - v_1) \, dx
<
0
.
$$
This contradiction shows that it is impossible to have $u_1 \ne v_1$ and $u_2 = v_2$. The case when $u_1 = v_1$ and $u_2 \ne v_2$ is excluded similarly. Hence we must have either $(u_1,u_2) = (v_1,v_2)$ or $(u_1,u_2) = ((u_1)_{\sigma_H(H)}, \, (u_2)_{\sigma_H(H)})$, as claimed. To complete the proof, suppose that $u_j$ is symmetric with respect to~$\partial H$ for some $j \in \{\, 1,2 \,\}$. For instance, suppose that $u_2$ is symmetric, the other case being analogous. Then $u_2 = v_2$, and the preceding argument shows that $u_1 = v_1$. Now we replace the half-space~$H$ with $\sigma_H(H)$, and we apply the same reasoning again, thus proving that $u_1 = (u_1)_{\sigma_H(H)}$, hence $u_1$ is symmetric. The proof is complete.
\end{proof}

\begin{remark}
We note that if $u=u_H$, then either $u$ is symmetric with respect to the reflection at $\partial H$ or there is $x\in H$ such that $u(x)>u(\sigma_H(x))$. Moreover, there exist non radial functions $u \colon \Omega \to \mathbb R$, defined in a radial set\/~$\Omega$, polarized with respect to every half-space~$H$. A two-dimensional example is given by $\Omega = B_1(0) \subset \mathbb R^2$ and $u(x_1,x_2) = x_1 \, (1 - |x|^2)$.
\end{remark}

\subsection{Proof of the existence of solutions with axial symmetry}

In the following, we finish the proof of Theorem \ref{thm:main3}. For this we assume $k_0$ satisfies \eqref{thm3:assumption1} with $c>0$ and $0<s\leq \sigma<1$. Let $\Omega$ be an open, bounded, radial domain in~$\mathbb R^N$, $N\geq 2$, and let $1<q<\frac{N}{\, N-2s \,}$. Moreover, we let $a_1,a_2\in L^{\infty}(\R)$ with $a_1\neq a_2$ and $\|a_i^+\|_{L^{\infty}(\R)}<\lambda_1(\Omega)$.

\begin{proof}[Proof of Theorem \ref{thm:main3} completed]
 Note that by Theorem \ref{thm:gradient} and Corollary \ref{positive} it follows that there are $u_1,u_2\in \cD_k(\Omega)$, $u_1,u_2>0$ satisfying \eqref{eq:system-basis2} with $u_1\neq u_2$. Moreover, by Theorem \ref{thm:polarization} and the radiality of $\Omega$ and $a_1,a_2$, it follows that for every half-space $H$ with $0\in \partial H$ we have either
\begin{itemize}
\item $u_1=(u_1)_H$ and $u_2=(u_2)_H$, or 
\item $u_1=(u_1)_{\sigma_H(H)}$ and $u_2=(u_2)_{\sigma_H(H)}$.
\end{itemize}
Hence, if either $u_1$ or $u_2$ is not radial, it follows that after a rotation ---and a renumbering if necessary--- the assumption \eqref{thm:main1-assumption2} is satisfied. Since clearly the right-hand sides of \eqref{eq:system-basis2} satisfy \eqref{thm:main1-assumption1c} the statement of Theorem \ref{thm:main3} follows from Theorem \ref{thm:main2} once we have shown that
\begin{equation}\label{claim:boundedness}
\text{$u_1$ and $u_2$ are bounded and continuous in $\Omega$.}
\end{equation}
The boundedness of the solution pair follows indeed by a standard iteration argument using the Sobolev embedding theorem. We give the details of this argument in the appendix (see Lemma \ref{boundedness0} and Corollary \ref{boundedness}). Having the boundedness of $u_1$ and $u_2$, the continuity of $u_1$ and $u_2$ in $\Omega$ follow e.g. from \cite{KM17}. Thus \eqref{claim:boundedness} holds and the statement of Theorem \ref{thm:main3} follows from Theorem~\ref{thm:main2} as mentioned before.
\end{proof}

\appendix

\section{On the polarization of a function in the nonlocal setting}

Recall the polarization of a function $u$ with respect to an open half space defined in \eqref{def:polarization}. Moreover, we use the notation of Section \ref{application}.

\medskip

In the next proposition we show that polarization reduces the energy, with special care to the equality case (see also \cite[Theorem~2]{Baernstein} and \cite[Proposition~8]{VW04}). In the proof we will need the following (somehow surprising) identity:

\begin{lemma}[Functional identity and inequality]\label{inequalities}
Let $H$ be a half-space in\/~$\mathbb R^N$, and let\/ $u \colon \allowbreak \mathbb R^N \to \mathbb R$ be any real-valued function. Define
\begin{align*}
f(x_1,x_2)
&=
u_H(x_1) \, u_H(x_2) + u_H(\sigma_H(x_1)) \, u_H(\sigma_H(x_2))
-
u(x_1) \, u(x_2) - u(\sigma_H(x_1)) \, u(\sigma_H(x_2)),
\\
\noalign{\medskip}
g(x_1,x_2)
&=
u_H(x_1) \, u_H(\sigma_H(x_2)) + u_H(\sigma_H(x_1)) \, u_H(x_2)
-
u(x_1) \, u(\sigma_H(x_2)) - u(\sigma_H(x_1)) \, u(x_2)
.
\end{align*}
For every $x_1,x_2 \in H$ we have $f(x_1,x_2) = -g(x_1,x_2) \ge 0$. Furthermore, $f(x_1,x_2) = 0$ if and only if
\begin{equation}\label{zeroth}
\big(u(x_1) - u(\sigma_H(x_1)\big)
\,
\big(u(x_2) - u(\sigma_H(x_2)\big)
\ge
0.
\end{equation}
\end{lemma}

\begin{proof}
Define $\xi_j = \frac1{\, 2 \,} \, \big(u(x_j) + u(\sigma_H(x_j))\big)$ and $\eta_j = \frac1{\, 2 \,} \, \big(u(x_j) - u(\sigma_H(x_j))\big)$, $j = 1,2$, so that
$$
\left\{
\begin{aligned}
u(x_j) &= \xi_j + \eta_j
\\
u(\sigma_H(x_j)) &= \xi_j - \eta_j
\end{aligned}
\right.
\qquad
\left\{
\begin{aligned}
u_H(x_j) &= \xi_j + |\eta_j|
\\
u_H(\sigma_H(x_j)) &= \xi_j - |\eta_j|
\end{aligned}
\right.
$$
where we have used the assumption that $x_1,x_2 \in H$. With this notation, we may write $f(x_1,x_2) = 2 \, |\eta_1 \, \eta_2| - 2 \, \eta_1 \, \eta_2$ and $g(x_1,x_2) = 2 \, \eta_1 \, \eta_2 - 2 \, |\eta_1 \, \eta_2|$, while inequality~\eqref{zeroth} reduces to $\eta_1 \, \eta_2 \ge 0$. The lemma follows.
\end{proof}

\begin{prop}[Polarization reduces the energy]\label{reduces-energy}
Let\/~$\Omega$ be an open set in\/~$\mathbb R^N$, and let $u \in \cD_k(\Omega)$. For every half-space $H \subset \mathbb R^N$ we have $\cE_k(u_H,u_H) \le \cE_k(u,u)$. Furthermore, if the kernel $k(z) = k_0(|z|)$ is given by a strictly decreasing function~$k_0$, then the equality $\cE_k(u_H,u_H) = \cE_k(u,u)$ holds if and only if either $u = u_H$, or $u = u_{\sigma_H(H)}$.
\end{prop}

\begin{proof}
We start by giving a convenient expression of $\cE_k(u,u)$. Since the integral is additive with respect to the domain of integration, we can split
$$
\cE_k(u,u)
=
\int_{H \times \mathbb R^N}
(u(x_1) - u(z))^2 \, k(x_1 - z) \, dx_1 \, dz\
+
\int_{\sigma_H(H) \times \mathbb R^N}
(u(x) - u(y))^2 \, k( x- y) \, dx \, dy
.
$$
The last integral, by the change of variables $x = \sigma_H(x_1)$ and $y = \sigma_H(z)$, satisfies
$$
\int_{\sigma_H(H) \times \mathbb R^N}
(u(x) - u(y))^2 \, k( x- y) \, dx \, dy\
=
\int_{H \times \mathbb R^N}
(u(\sigma_H(x_1)) - u(\sigma_H(z)))^2 \, k(x_1 - z) \, dx_1 \, dz
$$
and therefore we may write
$$
\cE_k(u,u)
=
\int_{H \times \mathbb R^N}
\Big(
(u(x_1) - u(z))^2
+
(u(\sigma_H(x_1)) - u(\sigma_H(z))^2
\Big)
\,
k(x_1 - z) \, dx_1 \, dz
.
$$
Let us repeat the argument once more: we split
\begin{align*}
\cE_k(u,u)
&=
\int_{H \times H}
\Big(
(u(x_1) - u(z))^2
+
(u(\sigma_H(x_1)) - u(\sigma_H(z))^2
\Big)
\,
k(x_1 - z) \, dx_1 \, dz
\\
\noalign{\medskip}
&+
\int_{H \times \sigma_H(H)}
\Big(
(u(x_1) - u(z))^2
+
(u(\sigma_H(x_1)) - u(\sigma_H(z))^2
\Big)
\,
k(x_1 - z) \, dx_1 \, dz
.
\end{align*}
Now in the first integral we write $x_2$ in place of~$z$, and in the last integral we let $x_2 = \sigma_H(z)$, thus obtaining
\begin{align*}
\cE_k(u,u)
&=
\int_{H \times H}
\Big(
(u(x_1) - u(x_2))^2
+
(u(\sigma_H(x_1)) - u(\sigma_H(x_2))^2
\Big)
\,
k(x_1 - x_2) \, dx_1 \, dx_2
\\
\noalign{\medskip}
&+
\int_{H \times H}
\Big(
(u(x_1) - u(\sigma_H(x_2)))^2
+
(u(\sigma_H(x_1)) - u(x_2)^2
\Big)
\,
k(x_1 - \sigma_H(x_2)) \, dx_1 \, dx_2
.
\end{align*}
By a similar procedure we also obtain
\begin{align*}
\cE_k(u_H,u_H)
&=
\int_{H \times H}\!
\Big(
(u_H(x_1) - u_H(x_2))^2
+
(u_H(\sigma_H(x_1)) - u_H(\sigma_H(x_2))^2
\Big)
\,
k(x_1 - x_2) \, dx_1 \, dx_2
\\
\noalign{\medskip}
&+
\int_{H \times H}\!
\Big(
(u_H(x_1) - u_H(\sigma_H(x_2)))^2
+
(u_H(\sigma_H(x_1)) - u_H(x_2)^2
\Big)
\,
k(x_1 - \sigma_H(x_2)) \, dx_1 \, dx_2
.
\end{align*}
To go further, observe that $u^2(x_j) + u^2(\sigma_H(x_j)) = u_H^2(x_j) + u_H^2(\sigma_H(x_j))$ for $j = 1,2$. Hence
\begin{align*}
{\textstyle\frac1{\, 2 \,}}
\,
\big(\cE_k(u,u) - \cE_k(u_H,u_H)\big)
&=
\int_{H \times H}
f(x_1,x_2)
\,
k(x_1 - x_2) \, dx_1 \, dx_2
\\
\noalign{\medskip}
&+
\int_{H \times H}
g(x_1,x_2)
\,
k(x_1 - \sigma_H(x_2)) \, dx_1 \, dx_2
\end{align*}
where $f$ and~$g$ are as in Lemma~\ref{inequalities}. Since $f = -g$, we may write
$$
{\textstyle\frac1{\, 2 \,}}
\,
\big(\cE_k(u,u) - \cE_k(u_H,u_H)\big)
=
\int_{H \times H}
f(x_1,x_2)
\,
\Big(
k(x_1 - x_2)
-
k(x_1 - \sigma_H(x_2))
\Big)
\,
dx_1 \, dx_2
.
$$
When the pair $(x_1,x_2)$ ranges in the domain of integration $H \times H$, the distance from~$x_1$ to~$x_2$ cannot be larger than the distance from~$x_1$ to~$\sigma_H(x_2)$ (see~\eqref{equality}). Since $k_0$ is monotone decreasing by assumption, it follows that $k(x_1 - x_2) - k(x_1 - \sigma_H(x_2)) \ge 0$, which implies $\cE_k(u_H,u_H) \le \cE_k(u,u)$ because $f$ is non-negative. To manage the special case when $k_0$ is strictly decreasing, we need the equality
\begin{equation}\label{equality}
|x_1 - \sigma_H(x_2)|^2 - |x_1 - x_2|^2 = 2 \, d_1 \, d_2
,
\end{equation}
where $d_j \ge 0$ denotes the distance from $x_j$ to~$\partial H$, $j = 1,2$. Equality~\eqref{equality} is established as follows. Let $\pi_j \in \partial H$ be the projection of~$x_j$ onto~$\partial H$, $j = 1,2$. Then by the Pythagorean theorem (see Figure~\ref{fig:2}) we have $|x_1 - \sigma_H(x_2)|^2 = (d_1 + d_2)^2 + |\pi_1 - \pi_2|^2$ as well as $|x_1 - x_2|^2 = (d_1 - d_2)^2 + |\pi_1 - \pi_2|^2$, and \eqref{equality} follows. Such an equality shows that $k(x_1 - x_2) - k(x_1 - \sigma_H(x_2)) > 0$ for all $x_1,x_2$ in the (open) half-space~$H$. But then the equality $\cE_k(u_H,u_H) = \cE_k(u,u)$ holds if and only if $f(x_1,x_2) = 0$ a.e.\ in~$H \times H$. By Lemma~\ref{inequalities}, this occurs if and only if \eqref{zeroth} holds a.e.\ in~$H \times H$. Clearly, if $u = u_H$ or $u(x) = u_H(\sigma_H(x))$ almost everywhere in\/~$\mathbb R^N$, then both factors in~\eqref{zeroth} have the same sign and therefore the inequality holds. Conversely, assume that \eqref{zeroth} holds true. Then $u$ may be symmetric with respect to~$\partial H$. Otherwise there exists a non-negligible set $X \subset H$ such that either $u(x_1) - u(\sigma_H(x_1)) > 0$ in~$X$, or $u(x_1) - u(\sigma_H(x_1)) < 0$ in~$X$. In the first case, \eqref{zeroth} implies that $u(x_2) - u(\sigma_H(x_2)) \ge 0$ a.e.\ in~$H$, hence $u = u_H$. In the second case, \eqref{zeroth} implies that $u = u_{\sigma_H(H)}$. The proof is complete.
\end{proof}

\begin{figure}[h]
\centering
\begin{picture}(200,120)(0,-5)
\put(0,0){$\bullet$} 
\put(-7,-6){$x_1$}
\put(180,25){$\bullet$} 
\put(185,20){$x_2$}
\put(180,100){$\bullet$} 
\put(185,95){$\sigma_H(x_2)$}
\put(-20,65){\line(1,0){225}} 
\put(210,63){$\partial H$}
\put(2.5,2){\line(0,1){63}} 
\put(0,68){$\pi_1$}
\put(182.5,27){\line(0,1){38}} 
\put(180,68){$\pi_2$}
\put(6,33){$d_1$}
\put(170,44){$d_2$}
\end{picture}
\caption{Finding $|x_1 - x_2|$ and $|x_1 - \sigma_H(x_2)|$}
\label{fig:2}
\end{figure}
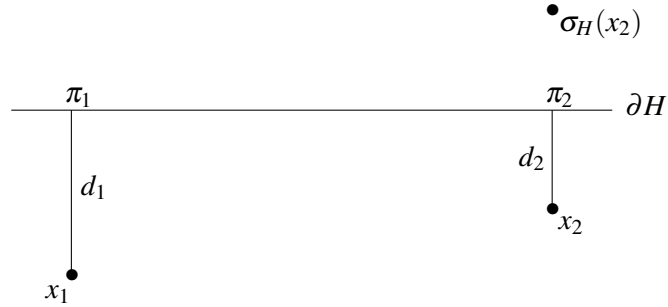

The proposition above, which deals with the energy functional, is used in combination with the following, which deals with the $L^q$-norm of the product of two given functions. Contrary to what one may expect, it turns out that polarization increases the norm:

\begin{prop}[On the $L^q$-norm of a product]\label{prop:increases}
Let\/ $\Omega\subset \R^N$ be an open, nonempty set, symmetric with respect to the boundary~$\partial H$ of some half-space~$H$. Take two non-negative functions $u,v \in L^q(\Omega)$ for some $q \in [1,\infty)$. Then
\begin{equation}\label{eq:increases}
\Vert u_H \, v_H \Vert_{L^q(\Omega)}
\ge
\Vert u \, v \Vert_{L^q(\Omega)}
.
\end{equation}
Furthermore, equality holds in~\eqref{eq:increases} if and only if
\begin{equation}\label{product:non-negative}
\big( u(x) - u(\sigma_H(x)) \big)
\,
\big( v(x) - v(\sigma_H(x)) \big)
\ge
0
\quad
\mbox{a.e.\ in\/ $\Omega$}
.
\end{equation}
\end{prop}

\begin{proof}
In order to prove~\eqref{eq:increases}, we split
$$
\int_\Omega u^q(x) \, v^{\, q}(x) \, dx\
=
\int_{\Omega \cap H} u^q(x) \, v^{\, q}(x) \, dx\
+
\int_{\Omega \cap \sigma_H(H)} u^q(y) \, v^{\, q}(y) \, dy
$$
and perform the change of variable $y = \sigma_H(x)$ in the last integral. Since $\Omega = \sigma_H(\Omega)$, we obtain
$$
\int_\Omega u^q(x) \, v^{\, q}(x) \, dx\
=
\int_{\Omega \cap H}
\Big(
u^q(x) \, v^{\, q}(x)
+
u^q(\sigma_H(x)) \, v^{\, q}(\sigma_H(x))
\Big)
\, dx
.
$$
By a similar procedure we also obtain
$$
\int_\Omega u^q_H(x) \, v^{\, q}_H(x) \, dx\
=
\int_{\Omega \cap H}
\Big(
u^q_H(x) \, v^{\, q}_H(x)
+
u^q_H(\sigma_H(x)) \, v^{\, q}_H(\sigma_H(x))
\Big)
\, dx
.
$$
Thus, it is enough to prove that for all $y \in \Omega \cap H$ we have
\begin{equation}\label{rearrangement}
u^q_H(x) \, v^{\, q}_H(x)
+
u^q_H(\sigma_H(x)) \, v^{\, q}_H(\sigma_H(x))
\ge
u^q(x) \, v^{\, q}(x)
+
u^q(\sigma_H(x)) \, v^{\, q}(\sigma_H(x))
,
\end{equation}
which is readily obtained from the rearrangement inequality \cite[(10.2.1)]{HLP}. To manage with the equality case, we prefer to let
$$
\left\{
\begin{aligned}
u^q(x) &= \xi_u + \eta_u
\\
u^q(\sigma_H(x)) &= \xi_u - \eta_u
\end{aligned}
\right.
\qquad
\left\{
\begin{aligned}
v^{\, q}(x) &= \xi_v + \eta_v
\\
v^{\, q}(\sigma_H(x)) &= \xi_v - \eta_v
\end{aligned}
\right.
$$
Thus, for $y \in \Omega \cap H$ we have
$$
\left\{
\begin{aligned}
u_H^q(x) &= \xi_u + |\eta_u|
\\
u_H^q(\sigma_H(x)) &= \xi_u - |\eta_u|
\end{aligned}
\right.
\qquad
\left\{
\begin{aligned}
v_H^{\, q}(x) &= \xi_v + |\eta_v|
\\
v_H^{\, q}(\sigma_H(x)) &= \xi_v - |\eta_v|
\end{aligned}
\right.
$$
and~\eqref{rearrangement} reduces to $|\eta_u \, \eta_v| \ge \eta_u \, \eta_v$, which obviously holds. Equality is achieved in~\eqref{eq:increases} if and only if $\eta_u \, \eta_v \ge 0$ a.e.\ in~$\Omega \cap H$, which is equivalent to~\eqref{product:non-negative}.
\end{proof}

\begin{remark}
If $u$ is symmetric with respect to~$\partial H$, for instance if $u$ is constant, then\/ \eqref{product:non-negative} holds for every~$v$.
\end{remark}

\section{On the boundedness of solutions}

In the following, $N\geq 2$, and we assume that $k_0$ satisfies \eqref{thm3:assumption1} for some $s,\gamma,\sigma\in(0,1)$ and $c>0$. Moreover, $\Omega$ is an open bounded set in $\R^N$ with Lipschitz boundary,  $\cE_k(u,v)$ and $\cD_k(\Omega)$ are defined as in Subsection~\ref{bilinearform}. For a related result with the fractional Laplacian, see also \cite[Lemma 2.3]{FMPSZ16}.
\begin{lemma}\label{boundedness0}
Let $A$ be a non-negative constant, and $1\leq q <\frac{\, 2_s \,}2$ with $2_s := \frac{2N}{\, N-2s \,}$.
If $u_1,u_2$ are two functions  in $\cD_k(\Omega)$ satisfying
\begin{equation}\label{eq:moser-iteration0}
	\left|\cE_k(u_i,\phi)\right|\leq \int_\Omega \Big( A \, |u_i| + |u_{3-i}|^q \, |u_i|^{q-1} \Big) \, \phi \, dx\quad\text{for all $\phi\in \cD_k(\Omega)$, $\phi\geq 0$, $i = 1,2$,}
\end{equation}
then $u_1,u_2\in L^\infty(\Omega)$.
\end{lemma}

\begin{remark}
Since $u_1,u_2,\varphi\in \cD_k(\Omega)$, it follows immediately by the Sobolev embedding theorem that $u_1,u_2,\varphi\in L^m(\Omega)$ for every $m \in [1,2_s]$. Since $q \le \frac{\, 2_s \,}2$ the integral in\/~\eqref{eq:moser-iteration0} converges. The strict inequality $q < \frac{\, 2_s \,}2$ is needed in the Moser iteration (see below).
\end{remark}

\begin{proof}[Proof of Lemma~\ref{boundedness0}]
We follow the idea of Moser's iteration presented in \cite{BP16} to show the claim.

\newcounter{n}\setcounter{n}{0}
\refstepcounter{n}\textit{Step \arabic{n}: Preliminaries.} Let $g \in W^{1,1}_{\rm loc}(\R)$ be nondecreasing and define
\begin{equation}
\label{defi:G}
G \colon \R\to \R,\quad G(t)=\int_0^t \sqrt{g'(\tau)}\ d\tau.
\end{equation}
Then we have (see also \cite[Lemma A.2]{BP16}) for $a,b\in \R$ using H\"older's inequality
\begin{align*}
(G(a)-G(b))^2=\Bigg(\int_{b}^aG'(t)\ dt\Bigg)^{\!\! 2}&\leq |b-a| \int_{\min\{a,b\}}^{\max\{a,b\}}G'(t)^2\ dt= |b-a| \int_{\min\{a,b\}}^{\max\{a,b\}}g'(t)\ dt\\
\noalign{\smallskip}
&=(a-b)(g(a)-g(b)).
\end{align*}
Hence, if $g\colon\R\to\R$ is a nondecreasing Lipschitz function, that is, we have for some $L_g>0$, $|g(a)-g(b)|\leq L_g \, |a-b|$ for all $a,b\in \R$, it follows that $\cE_k(g(v),g(v))\leq L_g^2 \, \cE_k(v,v)$ and
\begin{equation}
\label{estimate-below-start}
\cE_k(G(v),G(v))\leq \cE_k(v,g(v))\leq L_g \, \cE_k(v,v)\quad\text{ for all $v\in \cD_k(\Omega)$.}
\end{equation}
In particular, we see that both $g(v)$ and $G(v)$ belong to $\cD_k(\Omega)$.

\bigskip

\refstepcounter{n}\textit{Step \arabic{n}: A convenient Lipschitz function.} To apply Moser's iteration, define for $L > 0$, $r\geq 2$ the Lipschitz function 
\[
g\colon\R\to\R, \quad 
g(t)
=
\left\{\begin{aligned} &\ 0 && t\leq 0,\\
\noalign{\smallskip}
&\frac{t^{r-1}}{\, r-1 \,} && 0<t<L, \\
\noalign{\smallskip}
&\frac{L^{r-1}}{\, r-1 \,} && t\geq L.\end{aligned}
\right.
\]
Then $g'(t)=0=G(t)$ for $t< 0$, where $G$ is defined as in \eqref{defi:G}, and for $t>0$ we have
\[
g'(t)=\left\{\begin{aligned} &t^{r-2} && 0<t<L, \\
& 0 && t> L,\end{aligned}\right.
\quad\text{and hence}\quad 
G(t)
=
\left\{\begin{aligned} & \frac{\, 2 \, t^\frac{\, r \,}2 \,}r && 0<t<L, \\
\noalign{\smallskip}
& \frac{\, 2 \, L^\frac{\, r \,}2 \,}r  && t\geq L.\end{aligned}
\right.
\]
By the definition of $g$, it follows that if $v\geq0$ then $g(v)\geq 0$.

\bigskip

\refstepcounter{n}\label{above}\textit{Step \arabic{n}: Energy estimate from above.} We perform a suitable truncation of the kernel~$k$ and the solutions~$u_i$: our purpose is to get rid of the linear term $A \, |u_i|$ in~\eqref{eq:moser-iteration0}, thus proving~\eqref{estimate-above-1}. Let 
$$
k_{\delta}:=\chi_{B_{\delta}(0)} \, k\quad\text{and}\quad j_{\delta}:=k-k_{\delta}.
$$
 Note that $k_{\delta}$ satisfies the same assumptions as $k$ and, in particular, we have $\cD_{k_{\delta}}(\Omega)=\cD_k(\Omega)\subset L^{2_s}(\Omega)$ for all $\delta>0$. Moreover, by our assumptions on $k_0$, we have $j_{\delta}\in L^1(\R^N)\cap L^2(\R^N)$ for all $\delta>0$ and $J_{\delta}:=\|j_{\delta}\|_{L^1(\R^N)}\to \infty$ for $\delta \to 0$, hence we can fix some $\delta>0$ such that
$$
J_{\delta}>A.
$$
With the Cauchy-Schwarz inequality and $v_{in}=(u_i-n)^+$ for $n\in \N_0$ we have, since $k_\delta = k - j_\delta$ and taking \eqref{eq:moser-iteration0} into account,
 \begin{align*}
 \cE_{k_\delta}&(u_i,g(v_{in}))=\cE_{k}(u_i,g(v_{in}))-J_\delta\int_{\Omega} u_i(x) \, g(v_{in}(x))\ dx +\int_\Omega g(v_{in}(x)) \int_{\R^N} u_i(y) \, j_\delta(x-y)\ dy\ dx\\
&\leq  \int_\Omega \Big\{ (A\,|u_i(x)|-J_\delta\, u_i(x)) +|u_{3-i}(x)|^q \, |u_i(x)|^{q-1}+\int_{\R^N} u_i(y) \, j_\delta(x-y)\ dy \Big\} \, g(v_{in}(x)) \, dx\\
&= \int_\Omega \Big\{ (A-J_\delta) \, u_i(x)+|u_{3-i}(x)|^q \, |u_i(x)|^{q-1}+\int_{\R^N} u_i(y) \, j_\delta(x-y)\ dy \Big\} \, g(v_{in}(x)) \, dx\\
 &\leq \int_\Omega \Big\{ (A-J_\delta) \, n+|u_{3-i}(x)|^q \, |u_i(x)|^{q-1}+\|u_i\|_{L^2(\R^N)} \, \|j_\delta\|_{L^2(\R^N)} \Big\} \, g(v_{in}(x)) \, dx,
 \end{align*}
where we have used that $u_i\geq n$ in the set $\{\, g(v_{in})>0 \,\}$. Since $J_\delta>A$ and $u_i\in L^2(\Omega)$, we can fix from now on some $n\in \N$ large such that $(A-J_\delta) \, n+\|u_i\|_{L^2(\R^N)} \, \|j_{\delta}\|_{L^2(\R^N)}\leq 0$, and therefore
\begin{equation}\label{estimate-above-1}
\cE_{k_\delta}(u_i,g(v_{in})) \leq \int_\Omega |u_{3 - i}(x)|^q \, |u_i(x)|^{q-1} \, g(v_{in}(x)) \ dx,\quad i = 1,2.
\end{equation}
Let $p \in [2_s,\infty)$ be such that $u_1,u_2 \in L^p(\Omega)$, and observe that for $x\in\{\, g(v_{in})>0 \,\}$ we have 
$$
|u_i(x)|^{q - 1}=u_i(x)^{q-1}  = (n + v_{in})^{q - 1} \le C_q \, (n^{q - 1} + v_{in}^{q - 1})
$$
for a suitable constant $C_q$. Hence, from~\eqref{estimate-above-1} we get
\begin{equation}\label{estimate-above-2}
 \cE_{k_\delta}(u_i,g(v_{in})) \leq C_q \, \|u_{3 - i}\|_p^q \, \Big( n^{q - 1} \, \|g(v_{in})\|_\kappa+\|v_{in}^{q - 1} \, g(v_{in})\|_\kappa \Big),
\end{equation}
where $\kappa = \frac{p}{\, p-q \,} \in (1,2)$ is the conjugate exponent to $\frac{p}{\, q \,}$.
Here and in the following, for $\beta\geq 1$ we let $\|\cdot\|_{\beta}=\|\cdot\|_{L^{\beta}(\Omega)}$.

\bigskip

\goodbreak\refstepcounter{n}\label{below}\textit{Step \arabic{n}: Energy estimate from below.} To estimate $\cE_{k_{\delta}}(u_i,g(v_{in}))$ from below, note that we have with Lemma \ref{vk-properties}, \eqref{estimate-below-start}, and with the Sobolev embedding $\cD_{k_\delta}(\Omega) \hookrightarrow L^{2_s}(\Omega)$
$$
\cE_{k_{\delta}}(u_i,g(v_{in}))=\cE_{k_{\delta}}(u_i-n, \, g(v_{in}))\geq \cE_{k_{\delta}}(v_{in},g(v_{in}))\geq \cE_{k_{\delta}}(G(v_{in}),G(v_{in}))\geq\varepsilon_0 \, \|G(v_{in})\|_{2_s}^2,
$$
where $\varepsilon_0 > 0$ is a suitable constant.

\bigskip

\goodbreak\refstepcounter{n}\textit{Step \arabic{n}: We show that $u_1,u_2 \in L^p(\Omega)$ for all $p \in [1,\infty)$.} Combining the above inequality with \eqref{estimate-above-2}, we have
 \[
 \|G(v_{in})\|_{2_s}^2\leq C \, \|u_{3 - i}\|_p^q \, \Big(\|g(v_{in})\|_\kappa+\|v_{in}^{q-1}g(v_{in})\|_\kappa\Big),
 \]
 where $C$ depends on $n$, $\Omega$, $k$, $\delta$, $q$. Hence, with the monotone convergence theorem, we have for $L\to \infty$ and for every $r\geq 2$
\[
 \frac{4}{\, r^2 \,} \, \|v_{in}\|_{2_s \frac{\, r \,}2}^r\leq \frac{C}{\, r-1 \,} \,  \|u_{3 - i}\|_p^q \left(\|v_{in}\|_{(r-1) \, \kappa}^{r-1}+\|v_{in}\|_{(r+q-2) \, \kappa}^{r+q-2}\right) \!.
\]
Here and in the sequel it is understood that the norms of $v_{in}$ may attain the value~$\infty$. However if the right-hand side is finite, then the left-hand side is also finite, and the inequalities hold. Furthermore, by H\"older's inequality the $L^\beta$-norm in the bounded domain~$\Omega$ dominates the norm in $L^\alpha(\Omega)$ for $\alpha \in [1,\beta]$ in the sense that
\begin{equation}\label{dominates}
\|f\|_\alpha
\le
|\Omega|^\frac{\, \beta - \alpha \,}{\alpha\beta}
\,
\|f\|_\beta
.
\end{equation}
Using~\eqref{dominates} with $f = v_{in}$, $\alpha = (r - 1) \, \kappa$ and $\beta = (r + q - 2) \, \kappa$ we get
$$
\|v_{in}\|_{(r-1) \, \kappa}
\le
|\Omega|^\frac{q - 1}{\, (r - 1) (r + q - 2) \kappa \,}
\,
\|v_{in}\|_{(r+q-2) \, \kappa}
=
|\Omega|^\frac{(q - 1)(p - q)}{\, (r - 1) (r + q - 2) p \,}
\,
\|v_{in}\|_{(r+q-2) \, \kappa}
.
$$
In view of the subsequent application, it is relevant that the coefficient $|\Omega|^\frac{(q - 1)(p - q)}{\, (r - 1) (r + q - 2) p \,}$ keeps bounded when $p,r \to \infty$. Thus, by suitably modifying the constant~$C$ introduced before we may write
\begin{align*}
\frac4{\, r^2 \,} \, \|v_{in}\|_{2_s \frac{\, r \,}2}^r&\leq \frac{C}{\, r-1 \,} \, \|u_{3 - i}\|_p^q \left(\|v_{in}\|_{(r+q-2) \, \kappa}^{r-1}+\|v_{in}\|_{(r+q-2) \, \kappa}^{r+q-2}\right)
\\
\noalign{\medskip}
&\leq \frac{2C}{\, r-1 \,} \, \|u_{3 - i}\|_p^q \, \max \left\{1,\ \|v_{in}\|_{(r+q-2) \, \kappa}^{r+q-2}\right\}
\!.
\end{align*}
The last inequality follows from the fact that $r - 1 \le r + q - 2$. Hence
\begin{align}
\|v_{in}\|_{2_s \frac{\, r \,}2}&\leq \Big(\frac{\, C \,}2 \, \frac{r^2}{\, r-1 \,} \Big)^{\!\!\frac{1}{\, r \,}} \, \|u_{3 - i}\|_p^\frac{q}r \, \max\Big\{ 1,\ \|v_{in}\|_{(r+q-2) \, \kappa}^{1+\frac{q-2}{r}} \Big\}\notag\\
\noalign{\smallskip}
&\leq (Cr)^\frac1r \, \|u_{3 - i}\|_p^\frac{\, q \,}r \, \max\Big\{ 1,\ \|v_{in}\|_{(r+q-2) \, \kappa}^{1+\frac{q-2}{r}} \Big\}
\notag\\
\noalign{\smallskip}
&\leq C' \, \|u_{3 - i}\|_p^\frac{\, q \,}r \, \max\Big\{ 1,\ \|v_{in}\|_{(r+q-2) \, \kappa}^{1+\frac{q-2}{r}} \Big\}, \label{iteration-start-short}
\end{align}
where $C'$ has the same dependencies as $C$ by using that $(Cr)^\frac1r \to 1$ for $r\to \infty$, so that we can bound this quantity independently of $r\in[2,\infty)$. Notice that $v_{in} \le u_i^+ \le n + v_{in}$ in~$\Omega$, hence $u_i^+$~belongs to some $L^p(\Omega)$ if and only if $v_{in}$ does. To manage with
$u_1^-,u_2^-$, note that assumption~\eqref{eq:moser-iteration0} continues to hold if we replace $u_i$ with $-u_i$. Hence, following the Steps \ref{above}, \ref{below} and the above argumentation with $w_{in}=(-u_i-n)^+$ in place of $u_i$ for $i=1,2$ and $n$ as above, we also find
 \begin{equation}
\|w_{in}\|_{2_s \frac{\, r \,}2}\leq C' \, \|u_{3 - i}\|_p^\frac{\, q \,}r \, \max\Big\{ 1,\ \|w_{in}\|_{(r+q-2) \, \kappa}^{1+\frac{q-2}{r}} \Big\}. \label{iteration-start-short2}
 \end{equation}
Clearly, for $p\in[1,\infty]$ we have $u_i\in L^p(\Omega)$ if and only if $w_{in},v_{in}\in L^p(\Omega)$. In order to use~\eqref{iteration-start-short} and \eqref{iteration-start-short2}  iteratively (for both $i = 1$ and $i = 2$), we start from $p_0 = 2_s$, $\kappa_0 = \frac{2_s}{\, 2_s - q \,}$ and $r_0$ such that
 \[
 (r_0+q-2) \, \kappa_0 = 2_s,\quad\text{i.e. $r_0:=2_s-2q+2>2$.}
 \]
Moreover, for $m\in \N$ we define $p_m = 2_s \, \frac{\, r_{m - 1} \,}2$, $\kappa_m = \frac{p_m}{\, p_m - q \,}$ and we let $r_m$ be obtained from $r_{m - 1}$ through the equality $2_s \, \frac{\, r_{m - 1} \,}2 = (r_m + q - 2) \, \kappa_m$. In other terms, we define $\nu = \frac{\, 2_s \,}2$ and
$$
r_m:= \nu \, r_{m-1}-2\,(q-1) = 2 \, \nu^{m+1} - 2 \, (q - 1) \sum_{k=0}^m \nu^k
=2\,\nu^{m+1} \, \frac{\, \nu - q \,}{\, \nu - 1 \,}
+
2 \, \frac{\, q - 1 \,}{\, \nu - 1 \,}.
$$
Since $1 \le q < \nu$ by assumption, it follows that $r_m \nearrow \infty$ for $m\to \infty$ and therefore $u_1,u_2 \in L^p$ for every $p \in [1,\infty)$. Indeed, using the notation introduced above, and letting $r = r_m$ in~\eqref{iteration-start-short} we obtain
\begin{align*}
\|v_{in}\|_{p_{m+1}} &\leq C' \, \|u_{3 - i}\|_{p_m}^\frac{\, q \,}{\, r_m \,} \, \max\Big\{ 1,\ \|v_{in}\|_{p_m}^{1+\frac{q-2}{r_m}} \Big\}\quad\text{for $i=1,2$, $m\in \N$ and}\\
\|w_{in}\|_{p_{m+1}} &\leq C' \, \|u_{3 - i}\|_{p_m}^\frac{\, q \,}{\, r_m \,} \, \max\Big\{ 1,\ \|w_{in}\|_{p_m}^{1+\frac{q-2}{r_m}} \Big\}\quad\text{for $i=1,2$, $m\in \N$}
\end{align*}
and we may inductively apply the inequalities above together with~\eqref{dominates} to prove $v_{in},w_{in}\in L^p(\Omega)$ for all $p\in[1,\infty)$, but then $u_1,u_2\in L^p(\Omega)$ for $p\in[1,\infty)$ as claimed.

\bigskip

\refstepcounter{n}\textit{Step \arabic{n}: Conclusion.} To show that indeed we have $u_1,u_2\in L^{\infty}(\Omega)$, we fix $i\in\{1,2\}$ and $p_0 > \frac{N}{\, 2s \,}$. Since the product $f_{i}:=|u_{3-i}|^q \, |u_i|^{q-1}$ belongs to $L^{p_0}(\Omega)$, from \eqref{estimate-above-1} we get in place of~\eqref{estimate-above-2} the estimate
$$
\cE_{k_\delta}(u_i,g(v_{in})) \leq \|f_{i}\|_{p_0} \, \|g(v_{in})\|_{p'_0}
,
$$
where $p'_0=\frac{p_0}{p_0-1}$ and $v_{in}=(u_i-n)^+$ as above. Using again Step~\ref{below} and the subsequent argument, we arrive at
\begin{equation}\label{iteration-start-short-b1}
\|v_{in}\|_{2_s\frac{r}{2}}\leq (C \, \|f_{i}\|_{p_0} \, r)^\frac1{r} \, \|v_{in}\|_{(r-1) \, p'_0}^{1 - \frac1r} \leq (C \, \|f_{i}\|_{p_0} \, r)^\frac1{r} \, \max\{\, 1,\ \|v_{in}\|_{(r-1) \, p'_0} \,\},
\end{equation}
where $C$ depends on $n$, $\Omega$, $k$, and $\delta$. In order to use~\eqref{iteration-start-short-b1} iteratively, we define $r_0 = 2$, $\nu = \frac{2_s}{\, 2 p'_0 \,} > 1$ and
\begin{equation}\label{geometric}
r_m := \nu r_{m-1} + 1 = \frac{\, 2\nu - 1 \,}{\nu - 1} \, \nu^m -\frac1{\nu - 1}\quad\text{for $m \in \mathbb N$}.
\end{equation}
Note that $\alpha_m := 2_s \, \frac{\, r_m \,}2 = (r_{m+1} - 1) \, p'_0$ for $m \in \mathbb N_0$, and $r_m\nearrow \infty$ for $m\to \infty$. Letting $r = r_{m+1}$ and $M \geq C \, \|f_{i}\|_{p_{0}}$ in~\eqref{iteration-start-short-b1} we get $\|v_{in}\|_{r_{m+1}} \leq (Mr_{m+1})^\frac1{r_{m+1}} \, a_m$, where $a_m =\max\{\, 1,\ \|v_{in}\|_{\alpha_m} \,\} \geq 1$. Without loss of generality we take $M\ge\frac12$, so that $Mr_{m+1} \ge Mr_0 \ge 1$, and therefore we may write $a_{m+1} \leq \max\Big\{\, 1,\, (Mr_{m+1})^\frac1{r_{m+1}} \, a_m \,\Big\} = (Mr_{m+1})^\frac1{r_{m+1}} \, a_m$ for $m \in \mathbb N_0$. Thus, by induction we obtain
$$
a_m \leq a_0 \, \prod_{j = 1}^m (Mr_j)^\frac1{r_j}
\leq a_0 \, \prod_{j = 1}^\infty (Mr_j)^\frac1{r_j}
.
$$
Using~\eqref{geometric}, it is readily seen that the infinite product in the right-hand side converges to a (finite) limit, hence we have
\[
\|v_{in}\|_{\infty}=\lim_{m\to\infty}\|v_{in}\|_{\alpha_m}<\infty.
\]
Since the argument above also holds with $w_{in}=(-u_i-n)^+$ in place of $v_{in}$ we conclude that $u_1,u_2\in L^{\infty}(\Omega)$.
\end{proof}

\begin{cor}\label{boundedness}
Let $u_1,u_2\in \cD_k(\Omega)$ satisfy the system \eqref{eq:system-gradient} for some $a_1,a_2\in L^{\infty}(\Omega)$ and\/ $q$ such that $1 \le q < \frac{N}{\, N - 2s \,}$. Then $u_1,u_2\in L^\infty(\Omega)$.
\end{cor}

\begin{proof}
	Let $A:=\max\{\|a_i\|_{L^{\infty}(\Omega)}\;:\; i=1,2\}$, then for $i=1,2$ and any $\phi\in \cD_k(\Omega)$, $\phi\geq 0$ we have
	\[
	\left|\cE_k(u_i,\phi)\right|\leq \int_\Omega \Big(c \, |u_i| + |u_{3-i}|^q \, |u_i|^{q-1} \Big) \, \phi\ dx,
	\]
so the statement follows from Lemma \ref{boundedness0}.
\end{proof}

\section*{Acknowledgements}\pdfbookmark[1]{Acknowledgements}{Acknowledgements} The first author is a member of the Grup\-po Na\-zio\-na\-le per l'A\-na\-li\-si Ma\-te\-ma\-ti\-ca, la Pro\-ba\-bi\-li\-t\`a e le lo\-ro Ap\-pli\-ca\-zio\-ni (GNAMPA) of the I\-sti\-tu\-to Na\-zio\-na\-le di Al\-ta Ma\-te\-ma\-ti\-ca (INdAM). This work is partially supported by the research project \textit{Integro-differential Equations and Non-Local Problems}, funded by Fon\-da\-zio\-ne di Sar\-de\-gna (2017). We thank Tobias Weth for the discussions concerning the truncation of kernels used in the proof of Lemma \ref{boundedness0}.

\bibliographystyle{amsplain}

\end{document}